\documentclass[english,final,a4paper,11pt,pdftex,reqno]{amsart}
\usepackage[hmargin=3cm,vmargin=3.5cm]{geometry}
\usepackage{amssymb,colonequals,graphicx,color,enumitem}
\usepackage[plainpages=false,hypertexnames=false,colorlinks=true,bookmarksopen=true,bookmarksopenlevel=1]{hyperref}
\newcommand{\EPS}{\ensuremath{\varepsilon}}

\newcommand{\RE}[1]{\ensuremath{\text{Re}\,#1}}

\newcommand{\BAR}[1]{\ensuremath{\overline{#1}}}

\newcommand{\WSTAR}{\ensuremath{\text{weak}^{\ast}}}

\newcommand{\DD}[1]{\ensuremath{\partial #1}}
\newcommand{\QEDEX}{\ensuremath{\lozenge}}
\newcommand{\EMPTY}{\ensuremath{\varnothing}}

\newcommand{\DET}[1]{\ensuremath{\mbox{det}\,{#1}}}

\newcommand{\DIM}[1]{\ensuremath{\mbox{dim}\,#1}}

\newcommand{\ES}{\ensuremath{\sigma_{\textup{ess}}}}
\newcommand{\PS}{\ensuremath{\sigma_{\textup{p}}}}
\newcommand{\KER}{\ensuremath{\mathcal{N}}}
\newcommand{\RG}{\ensuremath{\mathcal{R}}}
\newcommand{\NN}{\ensuremath{\mathbb{N}}}

\newcommand{\RR}{\ensuremath{\mathbb{R}}}
\newcommand{\CC}{\ensuremath{\mathbb{C}}}

\newcommand{\CONT}{\ensuremath{C}}

\newcommand{\CONTD}[1]{\ensuremath{C^{#1}}}
\newcommand{\LP}[1]{\ensuremath{L^{#1}}}

\newcommand{\BND}{\ensuremath{\mathcal{L}}}

\newcommand{\NBV}{\ensuremath{\mbox{NBV}}}

\newcommand{\XC}{\ensuremath{X_0}}

\newcommand{\LCM}{\ensuremath{\mathcal{W}^c_{loc}}}

\newcommand{\CMMAPC}{\ensuremath{\mathcal{H}}}
\newcommand{\SUN}[1]{\ensuremath{#1^{\odot}}}
\newcommand{\STAR}[1]{\ensuremath{#1^{\ast}}}
\newcommand{\SUNSTAR}[1]{\ensuremath{#1^{\odot\star}}}
\newcommand{\SUNSUN}[1]{\ensuremath{#1^{\odot\odot}}}
\newcommand{\STARSTAR}[1]{\ensuremath{#1^{\ast\ast}}}

\newcommand{\PAIR}[3][]{\ensuremath{\langle #2,#3 \rangle_{#1}}}

\newcommand{\DEF}{\ensuremath{\colonequals}}

\theoremstyle{plain}
\newtheorem{theorem}{Theorem}[]
\newtheorem{corollary}[theorem]{Corollary}
\newtheorem{lemma}[theorem]{Lemma}
\newtheorem{proposition}[theorem]{Proposition}

\theoremstyle{definition}
\newtheorem{definition}[theorem]{Definition}

\newtheorem{remark}[theorem]{Remark}

\newcommand{\Id}{\ensuremath{I}}
\newcommand{\rr}{\ensuremath{\mathbf{r}}}

\numberwithin{equation}{section}

\begin{document}

\title{On Local Bifurcations in Neural Field Models with Transmission Delays}

\author[S.A. van Gils]{Stephan A. van Gils}
\address{S.A. van Gils\\Department of Applied Mathematics\\ University of Twente\\ The Netherlands \and Mathematical Insitute\\ Utrecht University\\ The Netherlands}
\email{s.a.vangils@math.utwente.nl}

\author[S.G. Janssens]{Sebastiaan G. Janssens}
\address{S.G. Janssens\\Department of Applied Mathematics\\ University of Twente\\ The Netherlands \and Mathematical Insitute\\ Utrecht University\\ The Netherlands}
\email{sj@dydx.nl}

\author[Yu.A. Kuznetsov]{Yuri A. Kuznetsov}
\address{Yu.A. Kuznetsov\\ Mathematical Insitute\\ Utrecht University\\ The Netherlands \and Department of Applied Mathematics\\ University of Twente\\ The Netherlands}
\email{i.a.kouznetsov@uu.nl}

\author[S. Visser]{Sid Visser}
\address{Sid Visser\\Department of Applied Mathematics\\ University of Twente\\ The Netherlands}
\thanks{Sebastiaan Janssens and Sid Visser gratefully acknowledge support by The Netherlands Organization of Scientific Research (NWO) through grant 635.100.019: \emph{From Spiking Neurons to Brain Waves}.}
\email{s.visser-1@math.utwente.nl}

\subjclass[2000]{Primary 37L10; Secondary 47H20, 37L05, 37M20 \and 92C20}
\keywords{delay equation, neural field, Hopf bifurcation, numerical bifurcation analysis, normal form, dual semigroup, sun-star calculus}

\date{\today}

\dedicatory{Dedicated to Odo Diekmann, on the occasion of his $65^{\text{th}}$ birthday.}

\begin{abstract}
   Neural field models with transmission delay may be cast as abstract delay differential equations (DDE). The theory of dual semigroups (also called sun-star calculus) provides a natural framework for the analysis of a broad class of delay equations, among which DDE. In particular, it may be used advantageously for the investigation of stability and bifurcation of steady states. After introducing the neural field model in its basic functional analytic setting and discussing its spectral properties, we elaborate extensively an example and derive a characteristic equation. Under certain conditions the associated equilibrium may destabilise in a Hopf bifurcation. Furthermore, two Hopf curves may intersect in a double Hopf point in a two-dimensional parameter space. We provide general formulas for the corresponding critical normal form coefficients, evaluate these numerically and interpret the results.
\end{abstract}

\maketitle

\section{Introduction}\label{sec:introduction}
Spatial coarse graining of neural networks leads to so-called neural field models in which the average firing rates of underlying populations of neurons, as opposed to individual neuronal spikes, are considered. Such models have not changed substantially since the seminal work of Wilson and Cowan \cite{Wilson1972,Wilson1973}, Amari \cite{Amari1977} and Nunez \cite{Nunez1974}. Due to intrinsic delays of axons, synapses, and dendrites in the natural system, the role of delays in spatiotemporal dynamics of neural activity has received considerable attention \cite{Liley2002,Hutt2003,Hutt2005,Roxin2005,Venkov2007,Hutt2007,Hutt2008,Coombes2007,Coombes2009,Coombes2010}. Faugeras and coworkers investigated stability properties of stationary solutions using methods from functional analysis \cite{veltz2009local,Faye2010,Veltz2011}. A first step towards Hopf bifurcation is made in \cite{Veltz2011b}, where Hopf bifurcation curves are computed. In \cite{Ermentrout1980} the principle of linearised stability and the Hopf bifurcation to periodic orbits were studied in the absence of delays.
\par
To set the stage, we have in mind $p \ge 1$ populations consisting of neurons that occupy fixed positions in a non-empty, bounded, connected, open region $\Omega \subset \RR^n$. For each $i = 1,\ldots,p$ let $V_i(t,\rr)$ be the membrane potential at time $t$, averaged over those neurons in the $i$th population positioned at $\rr \in \Omega$.  These potentials are assumed to evolve in the absence of time dependent external stimuli according to the system of integro-differential equations
\begin{equation}
  \label{eq:neuralfield:7}
  \frac{\DD V_i}{\DD t}(t,\rr) = -\alpha_iV_i(t,\rr) + \sum_{j=1}^p{\int_{\Omega}{J_{ij}(\rr,\rr')S_j(V_j(t - \tau_{ij}(\rr,\rr'),\rr'))\,d\rr'}}
\end{equation}
for $i = 1,\ldots,p$. The intrinsic dynamics exhibit exponential decay with $\alpha_i > 0$ for $i = 1,\ldots,p$. The propagation delays $\tau_{ij}(\rr,\rr')$ measure the time it takes for a signal sent by a type-$j$ neuron located at position $\rr'$ to reach a type-$i$ neuron located at position $\rr$. For the definitions and interpretation of the real valued connectivities $J_{ij}$ and the positive, real valued 
synaptic activation functions $S_j$ appearing in \eqref{eq:neuralfield:7} we refer to \S 2 of \cite{Veltz2011}. 
\par
The aim of this paper is to demonstrate how general theory from the field of delay equations can be used successfully to analyse stability and bifurcation of equilibrium solutions of \eqref{eq:neuralfield:7}. For this we consider a specific class of delay equations of the form
\begin{equation}
  \label{eq:neuralfield:8}\tag{DDE}
  \left\{
    \begin{aligned}
      \dot{x}(t) &= F(x_t) && t \ge 0\\
      x(t) &= \phi(t) && t \in [-h,0]
    \end{aligned}
  \right.
\end{equation}
where $Y$ is a Banach space, $F: \CONT([-h,0];Y) \to Y$ is a smooth $Y$-valued function on the Banach space of continuous $Y$-valued functions equipped with the supremum norm,
\begin{displaymath}
  x_t(\theta) \DEF x(t + \theta) \quad \forall\,t \ge 0,\,\theta \in [-h,0]
\end{displaymath}
is the {\it history} at time $t \ge 0$ and $\phi \in \CONT([-h,0];Y)$ is an initial condition. The parameter $h \in (0,\infty)$ is a \emph{finite} delay. As the reader may have expected, the acronym DDE stands for \emph{delay differential equation}. 
\par
Systems of this type naturally extend the case of \emph{classical} DDE with $Y = \RR^n$ for which a rather complete dynamical theory based on perturbative calculus of dual semigroups \cite{Clement1987}, \cite{Clement1988}, \cite{Clement1989}, \cite{Clement1989b}, \cite{Diekmann1991} is available in \cite{Diekmann1995}. Recently it was understood that, from an abstract viewpoint, various apparently different classes of delay equations can be cast and analysed within the same functional analytic framework of dual perturbation theory, largely independently of the particulars of a certain class. It is only in the choice of the underlying function spaces and the spectral analysis that these details matter. In \cite{Diekmann2007} purely functional equations (also called renewal equations) as well as systems of renewal equations coupled to delay differential equations are investigated for the $\RR^n$-valued case and finite delay. In \cite{Diekmann2012} the analysis is extended to the case of infinite delay. In \cite{Diekmann2008} abstract (Banach space valued) renewal equations with infinite delay are considered. The forthcoming paper \cite{VanGils2012b} treats general aspects of abstract equations of the type \eqref{eq:neuralfield:8}.
\par
The outline of this paper is as follows: in \S \ref{sec:setting} we introduce the functional analytic setting and state the equivalence between the abstract delay equation and an abstract integral equation using sun-star calculus. We also state a linearization theorem. In \S \ref{sec:spectrum} we start with some general results on the resolvent and spectra, primarily based on \cite{Engel2000}. For a specific class of connectivity functions, i.e. finite sums of exponentials, and in one spatial dimension, we explicitly calculate the spectrum and the resolvent. It turns out that the point spectrum is determined by a determinant condition. In \S \ref{sec:cm} we give the normal form coefficients for the critical center manifold in case of Hopf and double Hopf bifurcation. This is applied in \S \ref{sec:numerics} to a scalar neural field equation with a bi-exponential connectivity function modelling an inverted Wizard hat. The system is discretised as in \cite{Faye2010} and the spectrum of the discretised system is compared with the true spectrum, showing convergence. We identify in the true spectrum a Hopf point and a double Hopf point. 
For both cases the normal form coefficients are computed, which allows us to identify the sub-type of the bifurcation at hand. The theoretical results are confirmed by numerical experiments. We end this paper in \S \ref{sec:neuralfield:discussion} with conclusions and an outlook on future work.
\par
Upon finishing this paper we encountered the online preprint \cite{Veltz2012}, addressing similar questions. We feel that there are enough substantial differences between the two papers to render both of them interesting. Moreover, we have reasons to believe that the choice $Y=L_2(\Omega)$ for the spatial state space made in \cite{Veltz2012} leads to non-trivial technical complications, see \S \ref{sec:setting:y} below. In this paper, we employ sun-star calculus, from which a number of general results is immediately available. The center manifold, for instance, was obtained in \cite{Diekmann1995} for the abstract integral equation, covering what we need here.


\section{Functional analytic setting}\label{sec:setting}
\subsection{Basic definitions and assumptions}\label{sec:setting:basics}
\newcommand{\BOmega}{\ensuremath{\BAR{\Omega}}}
It is rather straightforward to associate with \eqref{eq:neuralfield:7} a problem of the type \eqref{eq:neuralfield:8}, but see \S \ref{sec:setting:y}. To keep the setting as simple as possible, we focus on the single population case $p = 1$ when
(\ref{eq:neuralfield:7}) takes the form
\begin{equation}
  \label{eq:neuralfield:7-1}
  \frac{\DD V}{\DD t}(t,\rr) = -\alpha V(t,\rr) + \int_{\Omega} J(\rr,\rr')S(V(t - \tau(\rr,\rr'),\rr'))\,d\rr'
\end{equation}
For mathematical convenience we extend the spatial domain $\Omega$ by its boundary $\DD \Omega$ and work on $\BOmega \equiv \Omega \cup \DD \Omega$ with Lebesgue measure $|\BOmega| < \infty$. We formulate a number of basic hypotheses on the modelling functions appearing in \eqref{eq:neuralfield:7-1}. These will be tacitly assumed to hold throughout the remainder of this paper. More specific functional forms will be chosen in subsequent sections.
\newcommand{\Htau}{($\text{H}_{\tau}$)}
\newcommand{\HJ}{($\text{H}_J$)}
\newcommand{\HS}{($\text{H}_S$)}
\begin{enumerate}[itemsep=1ex,leftmargin=3em]
\item[\HJ] The {\bf connectivity kernel} $J \in \CONT(\BOmega \times \BOmega)$.
\item[\HS] The {\bf synaptic activation function} $S \in \CONTD{\infty}(\RR)$ and its $k$th derivative is bounded for every $k \in \NN_0$.
\item[\Htau] The {\bf delay function} $\tau \in \CONT(\BOmega \times \BOmega)$ is non-negative and not identically zero.
\end{enumerate}
From {\Htau} we see that $\tau$ is bounded on the compact set $\BOmega$. Hence we may set
\begin{displaymath}
  0 < h \DEF \sup\{\tau(\rr,\rr')\,:\, \rr,\rr' \in \BOmega\} < \infty
\end{displaymath}
Let $Y \DEF \CONT(\BOmega)$ be the Banach space of continuous real-valued functions on $\BOmega$ with norm
\begin{displaymath}
  \|y\| \DEF \sup_{\rr \in \Omega}{|y(\rr)|},
\end{displaymath}
We also set $X \DEF \CONT([-h,0];Y)$. When $\phi \in X$, $t \in [-h,0]$ and $\rr \in \Omega$ we will sometimes abuse notation and write $\phi(t,\rr)$ instead of $\phi(t)(\rr)$. On $X$ we have the norm
\begin{displaymath}
  \|\phi\| \DEF \sup_{t \in [-h,0]}{\|\phi(t,\cdot)\|}
\end{displaymath}
Define the nonlinear operator $G : X \to Y$ by
\begin{equation}
  \label{eq:neuralfield:15}
  G(\phi)(\rr) = \int_{\BOmega}{J(\rr,\rr')S(\phi(-\tau(\rr,\rr'),\rr'))\,d\rr'} \quad \forall\,\phi \in X,\,\forall\,\rr \in \BOmega
\end{equation}
The following lemma is standard, but in light of the difficulties pointed out in \S\ref{sec:setting:y} we provide a detailed proof.
\begin{lemma}
  \label{lem:setting:1}
  $G : X \to Y$ is well-defined by \eqref{eq:neuralfield:15}.
\end{lemma}
\begin{proof}
  Obviously, for any $\phi \in X$ the map 
  \begin{equation}
    \label{eq:setting:5}
    [-h,0] \times \BOmega \ni (t,\rr) \mapsto \phi(t,\rr) \in \RR
  \end{equation}
  is continuous. 
  \par
  Now, given $\phi \in X$ we consider for points $\rr, \BAR{\rr} \in \BOmega$,
  \begin{align*}
    |G(\phi)(\rr) - G(\phi)(\BAR{\rr})| \le& \Bigl|\int_{\BOmega}{[J(\rr,\rr') - J(\BAR{\rr},\rr')]S(\phi(-\tau(\rr,\rr'),\rr'))\,d\rr'} \Bigr|\\
    +& \Bigl|\int_{\BOmega}{J(\BAR{\rr},\rr')[S(\phi(-\tau(\rr,\rr'),\rr')) - S(\phi(-\tau(\BAR{\rr},\rr'),\rr'))]\,d\rr'} \Bigr|\\
    \le& C_S \int_{\BOmega}{|J(\rr,\rr') - J(\BAR{\rr},\rr')|\,d\rr'}\\
    +& C_J \int_{\BAR{\Omega}}{|S(\phi(-\tau(\rr,\rr'),\rr')) - S(\phi(-\tau(\BAR{\rr},\rr'),\rr'))|\,d\rr'}
  \end{align*}
  where $C_S > 0$ and $C_J > 0$ are constants bounding $S$ and $J$. Let $\EPS > 0$ be given. By the uniform continuity of $J$ on $\BOmega \times \BOmega$ there exists $\delta_J > 0$ such that the first integral does not exceed $|\BOmega|\EPS$ for all $\rr,\BAR{\rr} \in \BOmega$ satisfying $|\rr - \BAR{\rr}| \le \delta_J$. Regarding the second integral, the continuity of \eqref{eq:setting:5} and {\Htau} implies the continuity of
  \begin{equation}
    \label{eq:setting:4}
    \BOmega \times \BOmega \ni (\rr,\rr') \mapsto \phi(-\tau(\rr,\rr'),\rr') \in \RR    
  \end{equation}
  Let $I \subset \RR$ be a compact interval containing the range of \eqref{eq:setting:4}. Then $S$ is uniformly continuous on $I$. Hence there exists $\delta_S > 0$ such that $|S(u) - S(v)| \le \EPS$ for all $u,v \in I$ satisfying $|u - v| \le \delta_S$. Since \eqref{eq:setting:4} is uniformly continuous, there exists $\delta' > 0$ such that $|\rr - \BAR{\rr}| \le \delta'$ implies $|\phi(-\tau(\rr,\rr'),\rr') - \phi(-\tau(\BAR{\rr},\rr'),\rr')| \le \delta_S$ for all $\rr' \in \BOmega$. Consequently, if $|\rr - \BAR{\rr}| \le \delta'$ then the second integral does not exceed $|\BOmega|\EPS$.
\end{proof}
Using the definition (\ref{eq:neuralfield:15}) of the operator $G$, we see that studying  (\ref{eq:neuralfield:7-1})  
is equivalent to analyzing the following initial value problem
\begin{equation}\tag{NF}
  \label{eq:neuralfield:10}
  \left\{
    \begin{aligned}
      \dot{V}(t) &= -\alpha V(t) + G(V_t) && t \ge 0\\
      V(t) &= \phi(t) && t \in [-h,0]
    \end{aligned}
  \right.  
\end{equation}
where $V : [-h,\infty) \to Y$ is the unknown and $\phi \in X$ is the initial condition. Then \eqref{eq:neuralfield:10} is of the form \eqref{eq:neuralfield:8} when we define $F : X \to Y$ by 
\begin{equation}
  \label{eq:setting:6}
  F(\phi) \DEF -\alpha \phi(0) + G(\phi) \qquad \forall\,\phi \in X
\end{equation}
with $G$ given by \eqref{eq:neuralfield:15}. The notion of a solution of \eqref{eq:neuralfield:8}, and consequently \eqref{eq:neuralfield:10}, is a direct generalisation of the solution concept for classical DDE.
\begin{definition}
  \label{def:setting:1}
  A function $x \in \CONT([-h,\infty);Y) \cap \CONTD{1}([0,\infty);Y)$ that satisfies \eqref{eq:neuralfield:8} is called a {\bf global solution} of \eqref{eq:neuralfield:8}. \hfill \QEDEX
\end{definition}
Sometimes we will omit the qualifier \emph{global} and simply speak of a \emph{solution} of \eqref{eq:neuralfield:8}. We conclude with a simple observation, which follows directly from the fact that {\HS} implies that $S$ satisfies a global Lipschitz condition.
\begin{lemma}
  \label{lem:setting:3}
  The operator $F : X \to Y$ defined by \eqref{eq:setting:6} is globally Lipschitz continuous.
\end{lemma}
\begin{proof}
  It suffices to show that $G$ satisfies a global Lipschitz condition. If $\phi, \BAR{\phi} \in X$ and $\rr,\rr' \in \BOmega$, then
  \begin{align*}
    |\phi(-\tau(\rr,\rr'),\rr') - \BAR{\phi}(-\tau(\rr,\rr'),\rr')| &\le \sup_{\rr'' \in \BOmega}{|\phi(-\tau(\rr,\rr'),\rr'') - \BAR{\phi}(-\tau(\rr,\rr'),\rr'')|}\\
    &\le \sup_{t \in [-h,0]}{\sup_{\rr'' \in \BOmega}{|\phi(t,\rr'') - \BAR{\phi}(t,\rr'')|  }} = \|\phi - \BAR{\phi}\|
  \end{align*}
  Hence we obtain
  \begin{align*}
    \|G(\phi)(\rr) - G(\BAR{\phi})(\rr)\| &\le |\BOmega|~\sup_{\BOmega \times \BOmega}{J}~ \sup_{\RR}{S'}~\|\phi - \BAR{\phi}\| \qquad \forall\,\rr \in \BOmega
  \end{align*}
  where the suprema are finite due to {\HJ} and {\HS}. 
\end{proof}

\subsection{Dual semigroups and DDE}\label{sec:setting:sunstar}
In this subsection we provide a very brief introduction to sun-star duality and its consequences for the analysis of \eqref{eq:neuralfield:10}. For a more complete treatment we refer to \cite{Diekmann1995} and, regarding the analysis of abstract DDE, the forthcoming paper \cite{VanGils2012b}. 
\par
In this subsection $Y$ will be a Banach space and $X \DEF \CONT([-h,0];Y)$. In conjunction with \eqref{eq:neuralfield:10} we will assume that $Y = \CONT(\BOmega)$. From an abstract point of view, solving a delay equation amounts to obtaining the future state of the system, say at time $t > 0$, from knowledge of the present state. This is done in two steps. First, the present state (a continuous function on the time segment $[-h,0]$) is extended to the interval $[-h,t]$. Next the part of this extension living on $[t-h,t]$ is shifted back to $[-h,0]$. Dual perturbation theory provides a systematic method to embed $X$ into a bigger Banach space, the so-called sun-star dual $\SUNSTAR{X}$, in which the extension and shifting operations are neatly separated. In broad lines, this works as follows.
\par
If the extension problem is trivial, i.e. if $F \equiv 0$ in \eqref{eq:neuralfield:8}, then the solution semigroup corresponding to \eqref{eq:neuralfield:8} is the shift semigroup $T_0$, defined as
\begin{equation}
  \label{eq:setting:8}
    (T_0(t)\phi)(\theta) =
  \begin{cases}
    \phi(t+\theta) &-h \le t + \theta \le 0\\
    \phi(0)        &\hphantom{-} 0 \le t + \theta 
  \end{cases}
  \quad \forall\,\phi\in X,\,t \ge 0,\,\theta \in [-h,0]
\end{equation}
Let $A_0$ be its infinitesimal generator. We represent $\STAR{X}$ by the space $\NBV([0,h];\STAR{Y})$ of functions $\eta : [0,h] \to \STAR{Y}$ of bounded variation, normalised such that $\eta(0) = 0$ and $\eta(t+) = \eta(t)$ for all $t \in (0,h)$. Elements of $X$ and $\STAR{X}$ are in duality via an abstract bilinear Riemann-Stieltjes integral. Since $X$ is not reflexive, the adjoint semigroup $\STAR{T}_0$ may not be strongly continuous on $\STAR{X}$. Let $\SUN{X} \subset \STAR{X}$ be the maximal subspace of strong continuity of $\STAR{T}_0$. It is easy to see that $\SUN{X}$ is positively $\STAR{T}_0$-invariant and, moreover, 
\begin{equation}
  \label{eq:setting:16}
  \SUN{X} = \BAR{D(\STAR{A}_0)}
\end{equation}
where $\STAR{A}_0$ is the adjoint of $A_0$. Let $\SUN{T}_0$ be the strongly continuous semigroup on $\SUN{X}$ obtained by restriction of $\STAR{T}_0$ to $\SUN{X}$. Its infinitesimal generator $\SUN{A}_0$ is precisely the part of $\STAR{A}_0$ in $\SUN{X}$, 
\begin{displaymath}
  D(\SUN{A}_0) = \{\SUN{\phi} \in D(\STAR{A}_0)\,:\, \STAR{A}_0\SUN{\phi} \in \SUN{X}\}, \qquad \SUN{A}_0\SUN{\phi} = \STAR{A}_0\SUN{\phi}
\end{displaymath}
In  \cite[Thm. 2.2]{Greiner1992} it is shown that $\SUN{X}$ may be identified with $\STAR{Y} \times \LP{1}([0,h];\STAR{Y})$ where the second factor is the space of Bochner integrable $\STAR{Y}$-valued functions on $[0,h]$. 
\par
Performing this construction once more, but now starting from the strongly continuous semigroup $\SUN{T}_0(t)$ on the Banach space $\SUN{X}$, we obtain the adjoint semigroup $\SUNSTAR{T}_0$ on the dual space $\SUNSTAR{X}$ and its strongly continuous restriction $\SUNSUN{T}_0$ to the positively invariant subspace $\SUNSUN{X} = \BAR{D(\SUNSTAR{A}_0)}$. The infinitesimal generator of $\SUNSUN{T}_0$ is again given by the part of $\SUNSTAR{A}_0$ in $\SUNSUN{X}$. Following \cite[\S 1.2]{Arendt2001} we suppose that it is not possible to represent $\SUNSTAR{X} = \STARSTAR{Y} \times \STAR{[\LP{1}([0,h];\STAR{Y})]}$ in terms of known functions or measures, since $\STARSTAR{Y}$ does not have the Radon-Nikodym property. However, the subspace $\SUNSUN{X}$ of strong continuity may be identified with $\CONT([-h,0],\STARSTAR{Y})$, see \cite[Thm. 3.11]{Greiner1992}. Of course this representation is only semi-explicit, since a representation for $\STARSTAR{Y}$ itself is unknown. The original space $X$ is canonically embedded into $\SUNSTAR{X}$ via $j : X \to \SUNSTAR{X}$ given by\footnote{In this paper we adopt the `postfix notation' for the action of a functional on a vector. That is, if $W$ is a Banach space with dual space $\STAR{W}$, $w \in W$ and $\STAR{w} \in \STAR{W}$, then $\PAIR{w}{\STAR{w}} \DEF \STAR{w}(w)$.}
\begin{equation}
  \label{eq:setting:18}
  \PAIR{\SUN{\phi}}{j\phi} \DEF \PAIR{\phi}{\SUN{\phi}} \quad \forall\,\phi \in X,\,\forall\,\SUN{\phi} \in \SUN{X}
\end{equation}
Since $Y$ is not reflexive, it follows that the range of $j$ must be a proper subspace of $\SUNSUN{X}$. This fact is expressed by saying that $X$ is \emph{not} sun-reflexive with respect to the shift semigroup $T_0$, a situation that contrasts the classical case $Y = \RR^n$. 
\par
We proceed to explain how \eqref{eq:neuralfield:8}, and consequently \eqref{eq:neuralfield:10}, fits into the above abstract context. Define $\delta \in \BND(\SUN{X},\STAR{Y})$ as 
\begin{equation}
  \label{eq:setting:13}
  \delta\SUN{\phi} \DEF \STAR{y} \qquad \forall\,\SUN{\phi} = (\STAR{y},g) \in \SUN{X}  
\end{equation}
Then $\STAR{\delta} \in \BND(\STARSTAR{Y},\SUNSTAR{X})$. Let $\ell \in \BND(Y,\SUNSTAR{X})$ be the restriction of $\STAR{\delta}$ to $Y$, viewed as a subspace of $\STARSTAR{Y}$. Explicitly,
\begin{equation}
  \label{eq:setting:12}
  \PAIR{y}{\delta \SUN{\phi}} = \PAIR{\SUN{\phi}}{\ell y} \qquad \forall\,y \in Y,\,\forall\,\SUN{\phi} \in \SUN{X}
\end{equation}
Define $R : X \to \SUNSTAR{X}$ by $R \DEF \ell \circ F$ with $F$ as in \eqref{eq:neuralfield:8}. The following lemma will prove to be useful in \S \ref{sec:normalforms:calc}. Observe that each $(y,f) \in Y \times \LP{\infty}([-h,0];Y)$ defines an element of $\SUNSTAR{X}$. Hence we may identify $Y \times \LP{\infty}([-h,0];Y)$ with a subspace of $\SUNSTAR{X}$.
\begin{lemma}
  \label{lem:setting:5}
  $R(\phi) = (F(\phi),0)$ for all $\phi \in X$. Hence $R$ maps into $Y \times \{0\}$. 
\end{lemma}
\begin{proof}
  Let $\phi \in X$ and write $R(\phi) = (\STARSTAR{y},\STAR{w}) \in \SUNSTAR{X}$ for certain $\STARSTAR{y} \in \STARSTAR{Y}$ and $\STAR{w} \in \STAR{[\LP{1}([0,h];\STAR{Y})]}$. Then, for any $\SUN{\phi} = (\STAR{y},g) \in \SUN{X}$, 
  \begin{equation}
    \label{eq:setting:14}
    \PAIR{\SUN{\phi}}{R(\phi)} = \PAIR{\STAR{y}}{\STARSTAR{y}} + \PAIR{g}{\STAR{w}}
  \end{equation}
  On the other hand, from the definition of $R$ we obtain
  \begin{equation}
    \label{eq:setting:15}
    \PAIR{\SUN{\phi}}{R(\phi)} = \PAIR{\SUN{\phi}}{\ell F(\phi)} = \PAIR{F(\phi)}{\delta\SUN{\phi}} = \PAIR{F(\phi)}{\STAR{y}}
  \end{equation}
  where in the second equality we used \eqref{eq:setting:12} and the third equality is due to \eqref{eq:setting:13}. By comparing \eqref{eq:setting:14} and \eqref{eq:setting:15} we see that $\STARSTAR{y}$ acts on $\STAR{y}$ by point evaluation in $F(\phi) \in Y$ and $\STAR{w} = 0$. Hence $R(\phi) = (F(\phi),0)$ and consequently $R$ maps into $Y \times \{0\}$.
 
\end{proof}
\begin{remark}
  \label{rem:setting:1}
  In the `classical' case where $Y = \RR^n$, the previous lemma shows that $R$ is a (possibly non-linear) operator of finite rank that takes values in the `point component' $Y$ only, see \cite[\S\S III.3 and VII.6]{Diekmann1995}. In the present setting with $\DIM{Y} = \infty$ we lose the former, but retain the latter property. \hfill \QEDEX
\end{remark}
We now consider the so-called {\it abstract integral equation} of the form
\begin{equation}\tag{AIE}
  \label{eq:setting:7}
    u(t) = T_0(t)\phi + j^{-1}\Bigl(\int_0^t{\SUNSTAR{T}_0(t - s)R(u(s))\,ds} \Bigr) \quad \forall\,t \ge 0
\end{equation}
where $\phi \in X$ is an initial condition, $u \in \CONT([0,\infty);X)$ is the unknown and the convolution integral is of {\WSTAR} Riemann type, see \cite[\S III.1]{Diekmann1995} and also \cite[Interlude 3.13 in Appendix II]{Diekmann1995}. In \cite{VanGils2012b} it is shown that this convolution integral takes values in the range of $j$. Consequently, the right-hand side of \eqref{eq:setting:7} is well-defined. The connection between \eqref{eq:neuralfield:8} and \eqref{eq:setting:7} is revealed in the following theorem. 
\begin{theorem}[Equivalence of \eqref{eq:neuralfield:8} and \eqref{eq:setting:7}]
  \label{thm:setting:1}
  Let $\phi \in X$ be given and let $R = \ell \circ F$ with $F \in \CONT(X,Y)$. The following two statements hold.
  \begin{enumerate}[itemsep=1ex,leftmargin=3em]
  \item[\textup{(i)}]
    Suppose that $u \in \CONT([0,\infty);X)$ satisfies \eqref{eq:setting:7}. Define $x : [-h,\infty) \to Y$ by $x_0 \DEF \phi$ and $x(t) = u(t)(0)$ for $t \ge 0$. Then $x$ is a global solution of \eqref{eq:neuralfield:8} in the sense of Definition \ref{def:setting:1}.
  \item[\textup{(ii)}] 
    Conversely, suppose that $x$ is a global solution of \eqref{eq:neuralfield:8}. Define $u : [0,\infty) \to X$ by $u(t) \DEF x_t$. Then $u \in \CONT([0,\infty);X)$ and $u$ satisfies \eqref{eq:setting:7}. 
  \end{enumerate}
\end{theorem}
It is routine to show that \eqref{eq:setting:7} admits unique global solutions on $[0,\infty)$ if $F$ is globally Lipschitz continuous. Thus, by Lemma \ref{lem:setting:3} and Theorem \ref{thm:setting:1} we find
\begin{corollary}
  \label{cor:setting:1}
  For any $\phi \in X$ problem \eqref{eq:neuralfield:10} has a unique global solution.
\end{corollary}
Of course, establishing well-posedness for \eqref{eq:neuralfield:10} does not require sun-star duality. Yet, it turns out that \eqref{eq:setting:7} is a very convenient tool in proving many standard results of dynamical systems for abstract DDE, such as the principle of linearised (in)stability, the existence of stable, unstable, and center manifolds, and theorems on local bifurcation. The following linearisation theorem is a direct generalisation of the corresponding result in the sun-reflexive case, see \cite[\S VII.5]{Diekmann1995} or \cite{Clement1989}.
\begin{theorem}[Linearisation at a steady state]
  \label{thm:setting:2}
  Let $F \in \CONTD{1}(X,Y)$ and $R = \ell \circ F$ and let $\Sigma$ be the strongly continuous non-linear semiflow on $X$ associated with \eqref{eq:setting:7}. Let $\hat{\phi} \in X$ be a steady state of $\Sigma$, i.e. $\Sigma(t)(\hat{\phi})=\hat{\phi}$ for all $t \geq 0$. The following statements are true.
  \begin{enumerate}[itemsep=1ex,leftmargin=3em]
  \item[\textup{(i)}]
    For each $t \ge 0$ the operator $\Sigma(t) : X \to X$ is continuously Fr\'echet differentiable in $\hat{\phi}$ with derivative $D\Sigma(t)(\hat{\phi}) \in \BND(X)$.
  \item[\textup{(ii)}]
    Upon defining $T(t) \DEF D\Sigma(t)(\hat{\phi})$ for each $t \ge 0$ one obtains a strongly continuous semigroup in $\BND(X)$. The domain of its generator $A$ is given by 
    \begin{equation}
      \label{eq:setting:11}
      D(A) = \{\phi \in X\,:\,\phi' \in X \text{ and } \phi'(0) = DF(\hat{\phi})\phi\}, \qquad A\phi = \phi'
    \end{equation}
  \item[\textup{(iii)}]
    For every $\phi \in X$ the function $T(\cdot)\phi \in \CONT([0,\infty),X)$ is the unique global solution of the linear abstract integral equation
    \begin{displaymath}
      T(t)\phi = T_0(t)\phi + j^{-1}\Bigl(\int_0^t{\SUNSTAR{T}_0(t - s)\ell DF(\hat{\phi})T(s)\phi\,ds} \Bigr)
    \end{displaymath}
\end{enumerate}
\end{theorem}
We observe that the above theorem produces a new strongly continuous semigroup $T$ on $X$ with generator $A$. For this semigroup we may likewise calculate the sun-star duality structure, just as we did for the shift semigroup $T_0$ defined by \eqref{eq:setting:8}. It turns out that the spaces $\SUN{X}$ and, consequently, $\SUNSTAR{X}$ are the \emph{same} for both semigroups. Indeed, if we put $B \DEF \ell \circ DF(\hat{\phi}) \in \BND(X,\SUNSTAR{X})$ and slightly abuse notation by writing $\STAR{B} \in \BND(\SUN{X},\STAR{X})$ for the restriction of the adjoint of $B$ to $\SUN{X}$, then just as in the sun-reflexive case \cite[\S III.2]{Diekmann1995} one proves that the adjoint of the generator $A$ of $T$ is given by
  \begin{displaymath}
    D(\STAR{A}) = D(\STAR{A}_0), \qquad \STAR{A}= \STAR{A}_0 + \STAR{B}
  \end{displaymath}
  By \eqref{eq:setting:16} the sun-duals of $X$ with respect to $T_0$ and $T$ are identical and may both be denoted by $\SUN{X}$. Moreover,
  \begin{equation}
    \label{eq:setting:10}
    D(\SUN{A}) = \{\SUN{\phi} \in D(\STAR{A})\,:\, \STAR{A}\SUN{\phi} \in \SUN{X}\}, \qquad \SUN{A}= \STAR{A}
  \end{equation}
  Let $\SUNSTAR{A} : D(\SUNSTAR{A}) \subseteq \SUNSTAR{X} \to \SUNSTAR{X}$ be its adjoint. For $\SUNSTAR{A}$ the situation is slightly more difficult than in the sun-reflexive case, because $D(\SUNSTAR{A}_0) \not\subseteq j(X)$. (Indeed, if it were true that $D(\SUNSTAR{A}_0) \subseteq j(X)$, then it would follow that $\SUNSUN{X} = \BAR{D(\SUNSTAR{A}_0)} \subseteq j(X)$ and $X$ would be sun-reflexive with respect to $T_0$, also see \cite[\S III.8]{Diekmann1995}.) The next lemma is sufficient for our purposes in \S \ref{sec:normalforms:calc}.
\begin{lemma}
  \label{lem:setting:4}
  If $\phi \in \CONTD{1}([-h,0];Y)$ then $j\phi \in D(\SUNSTAR{A})$ and $\SUNSTAR{A}j\phi = (0,\phi') + (DF(\hat{\phi})\phi,0)$. 
\end{lemma}

\subsection{Differentiability results}\label{sec:setting:diff}
The following two results concern the smoothness of the operator $G$ defined by \eqref{eq:neuralfield:15} and appearing in the right-hand side of \eqref{eq:neuralfield:10}. When $k = 1,2,\ldots$ we denote by $\BND_k(X,Y)$ the space of bounded $k$-linear operators from $X$ to $Y$. When $k = 1$ we write $\BND(X,Y)$ instead of $\BND_1(X,Y)$. For a review of differentiation in Banach spaces, we recommend \cite[Ch. 9]{Arbogast1999}.
\begin{lemma}
  \label{lem:neuralfield:9}
  The operator $G : X \to Y$ defined by \eqref{eq:neuralfield:15} is Fr\'echet differentiable with derivative $DG(\phi) \in \BND(X,Y)$ in the point $\phi \in X$ given by
  \begin{equation}
    \label{eq:setting:17}
    (DG(\phi)\psi)(\rr) = \int_{\BOmega}{J(\rr,\rr')S'(\phi(-\tau(\rr,\rr'),\rr'))\psi(-\tau(\rr,\rr'),\rr'))\,d\rr'}
  \end{equation}
  for all $\psi \in X$ and all $\rr \in \BOmega$.
\end{lemma}
\begin{proof}
  First we consider the operator $DG(\phi)$ defined by the right-hand side of \eqref{eq:setting:17}. Using standard methods as in the proof of Lemma \ref{lem:setting:1} one shows that $DG(\phi)\psi \in Y$ for all $\psi \in X$. These steps are omitted. As in the proof of Lemma \ref{lem:setting:3} we begin by noting that if $\rr, \rr' \in \BOmega$ then
  \begin{equation} 
    \label{eq:setting:1}
    |\psi(-\tau(\rr,\rr'),\rr')| \le \sup_{\rr'' \in \BOmega}{|\psi(-\tau(\rr,\rr'),\rr'')|} \le \sup_{t \in [-h,0]}{\sup_{\rr'' \in \BOmega}{|\psi(t,\rr'')|}} = \|\psi\|
  \end{equation}
  This implies that $\|(DG(\phi)\psi)\| \le M \|\psi\|$ where $M > 0$ is a constant depending on $\Omega$, $J$ and $S$. Hence $DG(\phi) \in \BND(X,Y)$. 
  \par
  Next we verify that $DG(\phi)$ is indeed the Fr\'echet derivative of $G$ at $\phi$. Introduce the shorthand notation $\phi^{\tau}(\rr,\rr') \DEF \phi(-\tau(\rr,\rr'),\rr')$. For $\eta \in X$ and $\rr \in \BOmega$,
  \begin{align*}
    G(\phi &+ \eta)(\rr) - G(\phi)(\rr) - [DG(\phi)\eta](\rr)\\
    &= \int_{\BOmega}{J(\rr,\rr')\left[S(\phi^{\tau}(\rr,\rr') + \eta^{\tau}(\rr,\rr')) - S(\phi^{\tau}(\rr,\rr')) - S'(\phi^{\tau}(\rr,\rr'))\eta^{\tau}(\rr,\rr')\right]\,d\rr'}
  \end{align*}
  Consider the integrand for fixed $\rr'$. It follows from the Mean Value Theorem that there exists $c = c(\phi,\eta,\rr,\rr') \in (0,1)$ such that 
  \begin{displaymath}
    S(\phi^{\tau}(\rr,\rr') + \eta^{\tau}(\rr,\rr')) - S(\phi^{\tau}(\rr,\rr')) = \eta^{\tau}(\rr,\rr') S'(\phi^{\tau}(\rr,\rr') + c \eta^{\tau}(\rr,\rr'))
  \end{displaymath}
  Consequently,
  \begin{align*}
    S(\phi^{\tau}(\rr,\rr') + \eta^{\tau}(\rr,\rr')) &- S(\phi^{\tau}(\rr,\rr')) - S'(\phi^{\tau}(\rr,\rr'))\eta^{\tau}(\rr,\rr')\\
    &= \left[S'(\phi^{\tau}(\rr,\rr') + c \eta^{\tau}(\rr,\rr')) - S'(\phi^{\tau}(\rr,\rr'))\right]\eta^{\tau}(\rr,\rr')
  \end{align*}
  Since $S'$ is uniformly continuous on compact intervals and $|\phi^{\tau}(\rr,\rr')| \le \|\phi\|$ and $|\eta^{\tau}(\rr,\rr')| \le \|\eta\|$ for all $\rr, \rr' \in \BOmega$, it follows that for every $\EPS > 0$ there exists $\delta > 0$ such that
  \begin{displaymath}
    |S'(\phi^{\tau}(\rr,\rr') + c \eta^{\tau}(\rr,\rr')) - S'(\phi^{\tau}(\rr,\rr'))| \le \EPS \qquad \forall\,\rr,\rr' \in \BOmega
  \end{displaymath}
  provided $\|\eta\| \le \delta$. Therefore, if $\|\eta\| \le \delta$ then
  \begin{displaymath}
    \|G(\phi + \eta) - G(\phi) - DG(\phi)\eta\| \le M \EPS \|\eta\|
  \end{displaymath}
  where $M > 0$ depends on $\Omega$ and $J$. This establishes differentiability. 
\end{proof}
\begin{proposition}
  \label{prop:setting:1}
  The operator $G$ defined by \eqref{eq:neuralfield:15} is in $\CONTD{\infty}(X,Y)$. For $k = 1,2,\ldots$ its $k$th Fr\'echet derivative $D^kG(\phi) \in \BND_k(X,Y)$ in the point $\phi \in X$ is given by
  \begin{displaymath}
    (D^kG(\phi)(\psi_1,\ldots,\psi_k))(\rr) = \int_{\BOmega}{J(\rr,\rr')S^{(k)}(\phi(-\tau(\rr,\rr'),\rr'))\prod_{i=1}^k{\psi_i(-\tau(\rr,\rr'),\rr')}\,d\rr'}
  \end{displaymath}
  for $\psi_1,\ldots,\psi_k \in X$ and $\rr \in \BOmega$.
\end{proposition}
\begin{proof}
  For $k = 1$ the statement reduces to Lemma \ref{lem:neuralfield:9}. Fix $k \ge 2$. We need to check that $D^{k-1}G : X \to \BND_{k-1}(X,Y)$ has Fr\'echet derivative $D^kG(\phi) \in \BND_k(X,Y)$ in the point $\phi \in X$. Again we remark that $D^kG(\phi)(\psi_1,\ldots,\psi_k) \in Y$ but we omit the proof. We begin by observing that $D^kG(\phi) \in \BND_k(X,Y)$. Indeed, by \eqref{eq:setting:1} we have 
  \begin{displaymath}
    \|D^kG(\phi)(\psi_1,\ldots,\psi_k)\| \le M \|\psi_1\| \cdot \ldots \cdot \|\psi_k\|
  \end{displaymath}
  for all $\psi_1,\ldots,\psi_k \in X$, where $M > 0$ is a constant depending on $\Omega$, $J$ and $S$.
  \par
  We conclude by verifying that $D^kG(\phi)$ is indeed the derivative of $D^{k-1}G$ at $\phi \in X$. Using the same shorthand notation as in the proof of Lemma \ref{lem:neuralfield:9}, we consider, for $\eta \in X$, $\psi = (\psi_1,\ldots,\psi_{k-1}) \in X^{k-1}$ with $\|\psi_i\| \le 1$ for all $i = 1,\ldots,k-1$ and $\rr \in \Omega$, 
  \begin{align*}
    (D^{k-1}G(\phi + \eta)\psi)(\rr) - (D^{k-1}G(\phi)\psi)(\rr) &- (D^kG(\phi)(\eta,\psi))(\rr)\\
    &= \int_{\BOmega}{J(\rr,\rr')R(\rr,\rr')\prod_{i = 1}^{k-1}{\psi_i^{\tau}(\rr,\rr')}\,d\rr'}
  \end{align*}
  where
  \begin{displaymath}
    R(\rr,\rr') \DEF S^{(k-1)}(\phi^{\tau}(\rr,\rr') + \eta^{\tau}(\rr,\rr')) - S^{(k-1)}(\phi^{\tau}(\rr,\rr')) - S^{(k)}(\phi^{\tau}(\rr,\rr'))\eta^{\tau}(\rr,\rr')
  \end{displaymath}
  Exactly as in the proof of Lemma \ref{lem:neuralfield:9} we may use the Mean Value Theorem and uniform continuity of $S^{(k)}$ on compact intervals to conclude that for each $\EPS > 0$ there exists $\delta > 0$ such that $\|\eta\| \le \delta$ implies $|R(\rr,\rr')| \le \EPS \|\eta\|$ for all $\rr,\rr' \in \BOmega$. Hence we have
  \begin{align*}
    \|D^{k-1}G(\phi + \eta)\psi - D^{k-1}G(\phi)\psi - D^kG(\phi)(\eta,\psi)\| \le M \EPS \|\eta\|
  \end{align*}
  provided $\|\eta\| \le \delta$, where $M > 0$ depends on $\Omega$ and $J$. 
\end{proof}

\subsection{Choosing the spatial state space}\label{sec:setting:y}
We believe that our choice for $Y = \CONT(\BOmega)$ made in \S \ref{sec:setting:basics} deserves some comments. In \cite{Faye2010,Veltz2011} the authors instead elect to work with the Hilbert space $Y = \LP{2}(\Omega)$. In our opinion this choice suffers from at least three mathematical complications.

\subsubsection*{The definition of $G$}\label{sec:setting:pointev}
It is no longer clear that $G$ is well-defined by \eqref{eq:neuralfield:15}. Namely, apart from square integrability one also needs to verify the following. If $\phi, \BAR{\phi} \in X$ and for all $t \in [-h,0]$ one has
  \begin{equation}
    \label{eq:setting:2}
    \phi(t,\rr') = \BAR{\phi}(t,\rr') \qquad \text{a.e. } \rr' \in \Omega
  \end{equation}
  (where a.e. stand for \emph{almost everywhere}, i.e. $\phi(t,\cdot)$ and $\BAR{\phi}(t,\cdot)$ represent the same element in $\LP{2}(\Omega))$ then this should imply that for almost all $\rr \in \Omega$ one has
  \begin{equation}
    \label{eq:setting:3}
    \phi(-\tau(\rr,\rr'),\rr') = \BAR{\phi}(-\tau(\rr,\rr'),\rr') \qquad \text{a.e. } \rr' \in \Omega
  \end{equation}
  There are bounded $\tau \in \CONT(\Omega \times \Omega)$ for which this implication fails. For example, let $\Omega = (0,1)$, write $x = \rr$ and $r = \rr'$ and let $\psi$ and $\BAR{\psi}$ be representatives of the same element in $\LP{2}(\RR)$ that differ in zero. If we define 
  \begin{displaymath}
    \phi(t,r) \DEF \psi(r + t), \qquad \BAR{\phi}(t,r) \DEF \BAR{\psi}(r + t) \qquad \forall\,r \in \Omega
  \end{displaymath}
  then $\phi, \BAR{\phi} \in X$ and \eqref{eq:setting:2} holds for all $t \in [-1,0]$. However, if $\tau(x,r) = r$, independent of $x \in \Omega$, then $\phi(-\tau(x,r),r) = \psi(0)$ and $\BAR{\phi}(-\tau(x,r),r) = \BAR{\psi}(0)$ which shows that for all $x = \rr \in \Omega$ we have \emph{in}equality in \eqref{eq:setting:3}. Clearly, this choice of $\tau$ is very unrealistic, but it does indicate a problem that needs to be addressed when one works with spaces of equivalence classes of measurable functions. It is not obvious that a more realistic choice such as $\tau(x,r) \DEF |x - r|$ does \emph{not} exhibit the above phenomenon.

\subsubsection*{First order Fr\'echet differentiability}
Even if we assume that the above problem can be solved satisfactorily by imposing additional (physiologically plausible) conditions on $\tau$, there remains the question of whether the first order Fr\'echet derivative of $G$ appearing in Lemma \ref{lem:neuralfield:9} maps $X$ into $Y$ when $Y = \LP{2}(\Omega)$. For the sake of simplicity, let us assume that $J(\rr,\rr') \equiv 1$ and $\phi \equiv 0$. Then the mapping
\begin{equation}
  \label{eq:setting:9}
  \Omega \ni \rr \mapsto \int_{\Omega}{\psi(-\tau(\rr,\rr'),\rr')\,d\rr'} \in \RR
\end{equation}
should be in $\LP{2}(\Omega)$ for all $\psi \in X$. This is not obvious. An attempt to prove this statement is contained in the proof following \cite[Lemma 3.1.1]{Faye2010}. The authors write, for $\rr \in \Omega$,
\begin{align*}
  \left(\int_{\Omega}{\psi(-\tau(\rr,\rr'),\rr')\,d\rr'} \right)^2 \le \int_{\Omega}{\psi^2(-\tau(\rr,\rr'),\rr')\,d\rr'} \le \sup_{t \in [-h,0]}{\int_{\Omega}{\psi^2(t,\rr')\,d\rr'}}
\end{align*}
The first estimate is by the Cauchy-Schwarz inequality. As it stands, the second estimate only seems to be valid under certain extra conditions on $\psi$ and / or $\tau$, since $\tau$ depends on the integration variable.

\subsubsection*{Higher order Fr\'echet differentiability}
In verifying second order differentiability we encounter problems similar to those pointed out above. Another complication appears in conjunction with derivatives of order three and higher. For instance, consider $k = 3$ in Proposition \ref{prop:setting:1}. Let $\Omega = (0,1)$ and write $x$ and $y$ for $\rr$ and $\rr'$. Define $\psi \in X$ by $\psi(t,r) \DEF r^{-\frac{1}{3}}$ for all $t \in [-h,0]$ and $r \in \Omega$. Then clearly $\psi \in X$ but the integral
\begin{displaymath}
  \int_{\Omega}{\psi^3(-\tau(x,r),r)\,dr}
\end{displaymath}
diverges for all $x \in \Omega$ so $D^3G(0)$ does not map $X$ into $Y$.

\subsubsection*{Which space to choose instead?}
It appears that the choice $Y = \LP{p}(\Omega)$ with $1 \le p < 1$ is not very fortunate. Moreover, from a biological point of view it is rather unclear why the membrane potentials should be merely $p$-integrable on $\Omega$ and not necessarily bounded. 
\par
Thus we are led to consider alternatives. Within the class of Hilbert spaces the Sobolev space $H^{k}(\Omega)$ comes to mind. By standard Sobolev embedding theory each element of $H^k(\Omega)$ has a (unique) continuous representative, provided $k \in \NN$ is sufficiently large (depending on the dimension of $\Omega$). Moreover, $H^k(\Omega)$ is a Banach algebra under mild conditions on $\Omega$ \cite[Thm. 5.23]{Adams1975}. However, for arbitrary $\phi \in X$ the mapping \eqref{eq:setting:9} cannot be expected to possess $k$ weak derivatives in $\LP{2}(\Omega)$.
\par
Other possibilities are $Y = \LP{\infty}(\Omega)$, $Y = B(\BOmega)$ and $Y = \CONT(\BOmega)$, where $B(\BOmega)$ is the Banach space of everywhere bounded, measurable functions on $\BOmega$. Note that the first two spaces differ in the sense that $\LP{\infty}(\Omega)$ consists of \emph{equivalence classes} of \emph{essentially} bounded, measurable functions on $\Omega$. The first choice satisfies all our needs, but it may potentially suffer from the problem indicated in \S \ref{sec:setting:pointev}. The second choice takes care of all the above technical complications but also introduces new ones. Most notably, the Arzel\'a-Ascoli theorem, used in  \S \ref{sec:spectrum:generalities}, does not hold in $B(\BOmega)$. The choice $Y = \CONT(\BOmega)$ seems to be fitting both from a modelling as well as from a technical perspective.


\section{Resolvents and spectra}\label{sec:spectrum}
Let $DG(\hat{\phi}) \in \BND(X,Y)$ be the Fr\'echet derivative of $G$ at the {\it steady state vector} $\hat{\phi} \in X$, i.e. $\hat{\phi}$ is independent of time (but possibly dependent on space) and 
\begin{equation}
  \label{eq:spectrum:21}
  -\alpha\hat{\phi} + G(\hat{\phi}) = 0 
\end{equation}
by \eqref{eq:neuralfield:10}. Using Lemma \ref{lem:neuralfield:9} we obtain
\begin{equation}
  \label{eq:spectrum:3}
  (DG(\hat{\phi})\phi)(\rr) = \int_{\BOmega}{J_0(\rr,\rr')\phi(-\tau(\rr,\rr'),\rr')\,d\rr'} \qquad \forall\,\phi \in X,\,\forall\,\rr \in \BOmega
\end{equation}
where 
\begin{equation}
  \label{eq:spectrum:10}
  J_0(\rr,\rr') \DEF J(\rr,\rr')S'(\hat{\phi}(-\tau(\rr,\rr'),\rr'))
\end{equation}
In this section we are interested in the spectral properties of the linear problem
\begin{equation}
  \label{eq:spectrum:4}
  \left\{
    \begin{aligned}
      \dot{V}(t) &= -\alpha V(t) + DG(\hat{\phi})V_t && t \ge 0\\
      V(t) &= \phi(t) && t \in [-h,0]
    \end{aligned}
  \right.
\end{equation}
with $\alpha > 0$, which is a special case of the problem  
\begin{equation}
  \label{eq:spectrum:1}
  \left\{
    \begin{aligned}
      \dot{x}(t) &= -\alpha x(t) + Lx_t && t \ge 0\\
      x(t) &= \phi(t) && t \in [-h,0]
    \end{aligned}
  \right.
\end{equation}
where $Y$ is a complex Banach space, $X = \CONT([-h,0];Y)$, $L \in \BND(X,Y)$ and $\alpha \in \CC$. 
\begin{remark}
  \label{rem:spectrum:1}
  For the spectral analysis of this section it is necessary to work in Banach spaces over $\CC$. So, whenever we discuss the spectral properties of \eqref{eq:spectrum:4}, we implicitly assume that the spaces $X$ and $Y$ and the operators acting between them have been complexified. In fact, one should also complexify the sun-star duality structure introduced in \S \ref{sec:setting:sunstar}. This task is not entirely trivial and rather tedious. Fortunately it has been carried out in \cite[\S III.7]{Diekmann1995}. \hfill \QEDEX
\end{remark}
In \S \ref{sec:spectrum:generalities} we make several standard observations on the structure of the spectrum of the generator of the strongly continuous semigroup solving \eqref{eq:spectrum:1}. Familiarity with the basics of spectral theory and semigroup theory is presumed, for which we recommend \cite[Ch.V]{Taylor1958} and \cite{Engel2000}. We would also like to mention the nice application-inspired paper \cite{Arino2006} for a detailed treatment of abstract linear DDE with bounded right-hand sides, partially in the context of Hale's \cite{Hale1977} formal duality approach. Some of our statements are similar to those found in \cite[\S 3.1]{Veltz2011}, but our approach (as well as the choice of state space, see the remarks in \S \ref{sec:setting:y}) is sometimes different. For instance, following \cite[Def. II.4.22]{Diekmann1995} and \cite[\S 4.1]{Arino2006}  we believe that the employment of Browder's (instead of Kato's) definition of the essential spectrum leads to somewhat simpler arguments.  
\par
In \S \ref{sec:spectrum:computation} we specialise to \eqref{eq:spectrum:4} and choose $Y = \CONT(\BOmega)$ and $L = DG(\hat{\phi})$. It is shown how to obtain explicit representations of resolvents and eigenvectors for a particular (but still rather general) choice of connectivity function $J$.

\subsection{Spectral structure}\label{sec:spectrum:generalities}
We recall from Theorem \ref{thm:setting:2} in \S \ref{sec:setting:sunstar} that the strongly continuous semigroup $T$ on $X$ corresponding to the global solution of \eqref{eq:spectrum:1} is generated by $A : D(A) \subset X \to X$ where
\begin{equation}
  \label{eq:spectrum:7}
  D(A) = \{\phi \in X\,:\,\phi' \in X \text{ and } \phi'(0) = -\alpha\phi(0) + L\phi\}, \qquad  A\phi = \phi'
\end{equation}
At this point we establish some standard notation. Let $S : D(S) \subset U \to U$ be a closed linear operator on a complex Banach space $U$. We denote by $\rho(S) \subset \CC$, $\sigma(A)$ and $\PS(A)$ the resolvent set, the spectrum and the point spectrum of $S$, respectively. When $z \in \rho(S)$ we write\footnote{It is customary to suppress the identity operator and write $\lambda - S$ instead of $\lambda I - S$.} $R(z,S) \DEF (z - S)^{-1}$ for the resolvent of $S$ at $z$. For any $z \in \CC$ we let $\RG(z - S)$ and $\KER(z - S)$ denote the range and the nullspace of $z - S$.
\par
The results in this subsection are rather easy consequences of the following generalisation of \cite[Thm. IV.3.1 and Cor. IV.3.3]{Diekmann1995}. It will turn out to be very convenient to employ tensor product $\otimes$ notation as introduced in \cite[p. 520]{Engel2000}. We recall the definition from there for the reader's convenience.
\begin{definition}
  \label{def:spectrum:1}
  Let $U,V$ be complex Banach spaces and let $\mathcal{F}(I,V)$ be a complex Banach space of $V$-valued functions defined on an interval $I \subseteq \RR$. Let $B \in \BND(U,V)$ and $g: I \to \CC$. If the map $g \otimes v: I \ni s \mapsto g(s) v \in V$ is in $\mathcal{F}(I,V)$ for all $v \in V$, then we define $g \otimes B: U \to \mathcal{F}(I,V)$ by
  \begin{displaymath}
    [(g \otimes B)u](s) \DEF (g \otimes Bu)(s) = g(s) Bu
  \end{displaymath}
for all $u \in U$ and $s \in I$. \hfill \QEDEX
\end{definition}
We also introduce some auxiliary operators.  For each $z \in \CC$ and $\theta \in [-h,0]$ we set $\EPS_{z}(\theta) \DEF e^{z \theta}$. With $L$ as in \eqref{eq:spectrum:1} we define
\begin{alignat}{4}
  L_{z} &\in \BND(Y),  &\qquad L_{z}f &\DEF L(\EPS_{z} \otimes f)\label{eq:spectrum:8}\\
  H_{z} &\in \BND(X),  &\qquad (H_{z}\phi)(\theta) &\DEF \int_{\theta}^0{e^{z(\theta - s)}\phi(s)\,ds}\nonumber\\
  S_{z} &\in \BND(X,Y), &\quad S_{z}\phi &\DEF \phi(0) + LH_{z}\phi\nonumber
\end{alignat}
for all $f \in Y$, $\phi \in X$ and $\theta \in [-h,0]$.
\begin{proposition}[{\cite[Prop. VI.6.7]{Engel2000}}]
  \label{prop:spectrum:1}
  For every $z \in \CC$ define $\Delta(z) \in \BND(Y)$ by
  \begin{equation}
    \label{eq:spectrum:2}
    \Delta(z) \DEF z + \alpha - L_{z}
  \end{equation}
  Then $\phi \in \RG(z - A)$ if and only if 
  \begin{equation}
    \label{eq:spectrum:5}
    \Delta(z)f = S_{z}\phi
  \end{equation}
  has a solution $f \in Y$ and, moreover, $z \in \rho(A)$ if and only if $f$ is also unique. If such is the case, then
  \begin{equation}
    \label{eq:spectrum:6}
    R(z,A)\phi = (\EPS_{z} \otimes \Delta(z)^{-1})S_{z}\phi + H_{z}\phi
  \end{equation}
  Furthermore, $S_{z}$ is surjective for every $z \in \CC$, so $\lambda \in \sigma(A)$ if and only if $0 \in \sigma(\Delta(\lambda))$. Finally, $\psi \in D(A)$ is an eigenvector corresponding to $\lambda$ if and only if $\psi = \EPS_{\lambda} \otimes q$ where $q \in Y$ satisfies $\Delta(\lambda)q = 0$.
\end{proposition}
\begin{corollary}
  \label{cor:spectrum:4}
  Let $z \neq -\alpha$. If $L_{z}$ is compact, then $\RG(z - A)$ is closed.
\end{corollary}
\begin{proof}
  From the part of Proposition \ref{prop:spectrum:1} regarding \eqref{eq:spectrum:5} we have $\phi \in \RG(z - A)$ if and only if $S_{z}\phi \in \RG(\Delta(z))$. From the theory of compact operators \cite[\S 5.5]{Taylor1958} it follows that $\Delta(z) = z + \alpha - L_{z}$ has closed range, since $z + \alpha \neq 0$. Now let $(\phi_n)_{n \in \NN}$ be a sequence in $\RG(z - A)$ converging to some $\phi \in X$. Then the sequence $(S_{z}\phi_n)_{n \in \NN}$ in $\RG(\Delta(z))$ converges to $S_{z}\phi \in \RG(\Delta(z))$, since $\RG(\Delta(z))$ is closed. Hence $\phi \in \RG(z - A)$. 
\end{proof}
\begin{remark}
  \label{rem:spectrum:2}
  For the particular case \eqref{eq:spectrum:4} with $Y = \CONT(\BOmega)$ and $L = DG(\hat{\phi})$, compactness of $L_{z}$ for each $z \in \CC$ follows easily from the Arzel\`a-Ascoli theorem since $L_{z}$ is a Fredholm integral operator with continuous kernel $J_0 e^{-z \tau}$. \hfill \QEDEX
\end{remark}
As $\DIM{Y} = \infty$ the shift semigroup $T_0$ on $X$ is no longer eventually compact. Consequently we need to consider the possibility that $\sigma(A)$ contains points that are \emph{not} isolated eigenvalues of finite type. 
\begin{definition}[{\cite[Def. 11]{Browder1961}}]
  \label{def:spectrum:2}
  The Browder essential spectrum $\ES(S)$ of a closed and densely defined operator $S: D(S) \subset U \to U$ consists of all $\lambda \in \sigma(S)$ for which at least one of the following three conditions holds:
\begin{enumerate}[itemsep=1ex,leftmargin=3em]
\item[(i)] $\lambda$ is an accumulation point of $\sigma(S)$;
\item[(ii)] $\RG(\lambda - S)$ is not closed;
\item[(iii)] $\bigcup_{k \ge 0}{\KER[(\lambda - S)^k]}$ has infinite dimension. \hfill \QEDEX 
\end{enumerate}
\end{definition}
We also recall that if $\lambda$ is in the point spectrum $\PS(S)$, then the closure of the subspace appearing in (iii) is the generalised eigenspace $M(\lambda,S)$ corresponding to $\lambda$. Its dimension $m_{\lambda}$ (which may be $\infty$) is the algebraic multiplicity of $\lambda$. If $m_{\lambda} < \infty$ then $\lambda$ is called an eigenvalue of finite type. If $m_{\lambda} = 1$ then $\lambda$ is called a simple eigenvalue. 
\begin{corollary}
  \label{cor:spectrum:3}
  Suppose $L_{z}$ is compact for all $z \neq -\alpha$. Then $\ES(A) \subseteq \{-\alpha\}$. Moreover, $\sigma(A) \setminus \{-\alpha\}$ consists of poles of $R(\cdot,A)$. The order of $\lambda$ as a pole of $R(\cdot,A)$ equals the order of zero as a pole of $R(\cdot,\Delta(\lambda))$.
\end{corollary}
\begin{proof}
  Let $\lambda \in \sigma(A)$ and $\lambda \neq -\alpha$. Corollary \ref{cor:spectrum:4} implies that (ii) in Definition \ref{def:spectrum:2} (with $S = A$ and $U = X$) cannot be true. Since $\lambda + \alpha \neq 0$ is in $\sigma(L_{\lambda})$ and $L_{\lambda}$ is compact, it follows (again from general spectral theory, see e.g. \cite[\S 5.8]{Taylor1958}) that $\lambda + \alpha$ is a pole of $R(\cdot,L_{\lambda})$, say of order $k \ge 1$. If we can prove that $\lambda$ is a pole of order $k$ of $R(\cdot,A)$, then we are done. Indeed, it then follows that in particular $\lambda$ is isolated in $\sigma(A)$, so (i) in Definition \ref{def:spectrum:2} cannot hold. By \cite[Thm.5.8-A]{Taylor1958} the same is true for (iii). 
\par
Let us therefore prove that $\lambda$ is a pole of order $k$ of $R(\cdot,A)$. First we remark that the map 
\begin{equation}
  \label{eq:appendix_laurent:2}
  \CC \ni z \mapsto L_{z} \in \BND(Y)
\end{equation}
is continuous at $\lambda$. The proof of this fact is standard and has been omitted. If $z$ is in $\rho(A)$ then $z + \alpha \in \rho(L_{z})$ and
\begin{displaymath}
  \Delta(z)^{-1} = (z + \alpha - L_{z})^{-1} = \bigl[z + \alpha - (L_{\lambda} + (L_{z} - L_{\lambda})) \bigr]^{-1}
\end{displaymath}
A continuity property of the resolvent \cite[Theorem IV.3.15]{Kato1966} together with the continuity of \eqref{eq:appendix_laurent:2} at $\lambda$ implies that for $z$ sufficiently close to $\lambda$ we have $z + \alpha \in \rho(L_{\lambda})$ and
\begin{displaymath}
  \Delta(z)^{-1} = (z + \alpha - L_{\lambda})^{-1} + o(|\lambda - z|) \qquad \text{as } z \to \lambda
\end{displaymath}
where $o(|\lambda - z|)$ denotes a term that vanishes as $z \to \lambda$. By \eqref{eq:spectrum:6} we see that for $z$ sufficiently close to $\lambda$,
\begin{equation}
  \label{eq:appendix_laurent:3}
  R(z,A) = (\EPS_{z} \otimes (z + \alpha - L_{\lambda})^{-1})S_{z} + H_{z} + o(|\lambda - z|)
\end{equation}
where it was used that $\|S_{z}\|$ remains bounded as $z \to \lambda$, which can easily been seen from \eqref{eq:spectrum:8} (with $\lambda$ replaced by $z$). This already establishes that $\lambda$ is an isolated singularity of $R(\cdot,A)$. To conclude the proof we recall that $\lambda + \alpha$ is a pole of $R(\cdot,L_{\lambda})$ of order $k \ge 1$. Hence $\lambda$ itself is a pole of order $k$ of the mapping 
\begin{displaymath}
  \CC \ni z \mapsto (z + \alpha - L_{\lambda})^{-1} \in \BND(Y)  
\end{displaymath}
The result now follows from \eqref{eq:appendix_laurent:3} since $\CC \ni z \mapsto H_{z} \in \BND(X)$ is analytic in $z = \lambda$ and the zero-order term in the power series expansion of $\CC \ni z \mapsto S_{z} \in \BND(X,Y)$ at $z = \lambda$ does not vanish, as is easily checked.

\end{proof}
Hence, although for the application to \eqref{eq:spectrum:4} that we have in mind essential spectrum exists in the form of the exceptional point $-\alpha < 0$, it is properly contained in the left half-plane and therefore rather harmless. This situation seems to be quite common in DDE arising in population dynamics, see the remark in \cite[p. 321]{Arino2006}. As a pleasant consequence, most of the results in \cite[\S IV.2]{Diekmann1995} have immediate analogues in the present setting. We will limit ourselves to the statement of two such results that are also important for the application of center manifold theory in \S \ref{sec:cm}. 
\begin{lemma}[{\cite[Thm. VI.6.6 and Corollary IV.3.11]{Engel2000}}]
  \label{lem:spectrum:1}
  The semigroup $T$ generated by $A$ is norm continuous for $t > h$. Consequently $\omega_0 = s(A)$, where $\omega_0$ is the growth bound of $T$ and $s(A)$ is the spectral bound of $A$.
\end{lemma}
The above lemma implies that, for the linear problem \eqref{eq:spectrum:4}, the (in)stability of the zero solution may be inferred from the location of the poles of $R(\cdot,A)$ in the complex plane. More precisely, we have the following result, which is a direct analogue of \cite[Thm. IV.2.9]{Diekmann1995}.
\begin{proposition}
  \label{prop:spectrum:2}
  Suppose $\beta > -\alpha$. Let 
  \begin{displaymath}
    \Lambda = \Lambda(\beta) \DEF \{\lambda \in \sigma(A)\,:\,\RE{\lambda} > \beta\}
  \end{displaymath}
  and let $P_{\Lambda} \in \BND(X)$ be the spectral projection associated with $\Lambda$, see \cite[\S 5.7]{Taylor1958}. Then
  \begin{displaymath}
    X = \RG(P_{\Lambda}) \oplus \RG(I - P_{\Lambda})     
  \end{displaymath}
  where the first summand is finite dimensional, the second summand is closed and both summands are positively $T$-invariant. Moreover, there exist $K > 0$ and $\EPS > 0$ such that
  \begin{alignat}{4}
    \|T(t)P_{\Lambda}\| &\le K e^{(\beta + \EPS)t}\|P_{\Lambda}\| &\qquad &\forall\,t \le 0\label{eq:spectrum:20}\\
    \|T(t)(I - P_{\Lambda})\| &\le K e^{(\beta + \EPS)t}\|I - P_{\Lambda}\| &\qquad &\forall\,t \ge 0\nonumber
  \end{alignat}
\end{proposition}
We observe that $T(t)P_{\Lambda}$ is well-defined in \eqref{eq:spectrum:20} for all $t \le 0$, since  $T(t)$ extends uniquely to a \emph{group} on the finite-dimensional range of $P_{\Lambda}$.
\par
The extension of the above decomposition and exponential estimates to $\SUNSTAR{X}$ proceeds exactly as in \cite[p.100 - 101]{Diekmann1995}.

\subsection{Explicit computations}\label{sec:spectrum:computation}
In the remainder of this section we consider a homogeneous neural field with transmission delays due to a finite propagation speed of action potentials as well as a finite, fixed delay $\tau_0 \ge 0$ caused by synaptic processes. Space and time are each rescaled such that $\BOmega = [-1, 1]$ and the propagation speed is $1$. This yields
\begin{equation}
  \label{eq:spectrum:19}
  \tau(x,r) = \tau_0 + |x-r| \qquad \forall\,x,r \in \BOmega
\end{equation}
For the connectivity function we take a linear combination of $N \ge 1$ exponentials,
\begin{equation}
  \label{eq:spectrum:9}
  J(x,r) = \sum_{i=1}^{N} \hat{c}_i e^{-\mu_i |x-r|} \qquad \forall\,x,r \in \BOmega 
\end{equation}
where
\begin{displaymath}
  \hat{c}_i\in\mathbb{C} \text{ with } \hat{c}_i\neq0, \qquad \mu_i\in\mathbb{C} \text{ with } \mu_i \neq \mu_j \text{ for } i\neq j
\end{displaymath}
(As the number $N$ of exponentials remains fixed, we suppress it in our notation.) In addition to {\HS} we also require here that $S(0) = 0$. We study the stability of a spatially homogeneous steady state $\hat{\phi} \equiv 0$ by analysing the spectrum of the linearised system \eqref{eq:spectrum:4}. Following \eqref{eq:spectrum:10} we incorporate $S'(0)$ into the connectivity function,
\begin{displaymath}
  J_0(x,r) = \sum_{i=1}^{N} c_i e^{-\mu_i |x-r|}, \quad c_i = S'(0)\hat{c}_i
\end{displaymath}
In order to avoid overly convoluted notation we henceforth write $J$ instead of $J_0$. Assuming the  form \eqref{eq:spectrum:9}, in the next two subsections we explicitly compute the point spectrum $\PS(A)$ with $A$ as in \eqref{eq:spectrum:7} as well as the resolvent operator $R(\lambda,A)$ for $\lambda \in \rho(A)$. 

\subsection{Characteristic equation}\label{sec:spectrum:chareq}
In this example the operator $\Delta(\lambda)$ introduced in \eqref{eq:spectrum:2} is given by
\begin{equation}
  \label{Integral01}
  (\Delta(\lambda)  q)(x)= 
  (\lambda+\alpha) q(x) - \int_{-1}^1{J(x,r)e^{-\lambda\tau_0}e^{-\lambda|x-r|}q(r)\,dr} 
\end{equation}
for all $\lambda \in \CC$, $q \in Y$ and $x \in \BOmega = [-1,1]$. We let
\begin{equation}
  \label{eq:spectrum:13}
  k_i \DEF \lambda+\mu_i \qquad \forall\,i=1,\hdots,N
\end{equation}
and define for each $i = 1,\ldots,N$ the integral operator $K_i \in \BND(Y)$ by
\begin{displaymath}
  (K_i q)(x) = \int_{-1}^1{e^{-k_i|x-r|}q(r)\,dr}
\end{displaymath}
and set $(Kq)(x) \DEF [(K_1q)(x),\ldots,(K_Nq)(x)] \in \CC^N$. By introducing $c \DEF [c_1,\dots,c_N] \in \CC^N$, $\Delta(\lambda)$ is written as
\begin{equation}
  \label{eq:spectrum:11}
  \Delta(\lambda)q =(\lambda+\alpha)e^{\lambda\tau_0}q - (c\cdot Kq), \qquad (c \cdot Kq)(x) \DEF (c \cdot Kq(x))
\end{equation}
where $(a\cdot b) \DEF \sum_{i=1}^N{a_ib_i}$ is a pairing of two complex vectors $a = [a_1,\ldots,a_n]$ and $b = [b_1,\ldots,b_n]$. We solve the equation $\Delta(\lambda) q=0$ by formulating a linear ODE in terms of $ q$ by repetitive differentiation. For this purpose the next lemma is useful.
\begin{proposition}
  All solutions $q \in Y$ of the equation $\Delta(\lambda) q=0$ are in fact in $\CONTD{\infty}(\BOmega)$.
\end{proposition}
\begin{proof}
  The range of $K_i$ is contained in $C^1(\BOmega)$ and therefore any solution of the equation $\Delta(\lambda) q=0$ is an element of $C^\infty(\BOmega)$.
\end{proof}
Let $q \in \CONTD{2}(\BOmega)$. The first derivative of $\Delta(\lambda) q$ with respect to the spatial variable contains terms that involve integration over the intervals $[-1,x]$ and $[x,1]$. The second derivative has a nicer structure:
\begin{equation}
  \label{SecondDerivativeIdentity}
  D^2_x\Delta(\lambda) q = (\lambda+\alpha)e^{\lambda\tau_0} q^{(2)} + 2(c \cdot k)q - (ck^2 \cdot Kq)
\end{equation}
in which $ q^{(2)}$ denotes the second derivative of $q$ and for each $m \in \NN$ the vectors $k^m$ and $ck^m$ in $\CC^N$ have elements $k_i^m$ and $c_ik_i^m$, respectively, for $i=1,\ldots,N$. This identity allows for straightforward calculation of higher derivatives. For the following lemma we recall the definition in \eqref{eq:spectrum:13}.
\begin{lemma}
  \label{prp:UniqueK}
  The set $\mathcal{S} \DEF \{\lambda \in \CC\,:\,\exists i,j \in \{1,\ldots,N\}, i \neq j, \text{ s.t. } k_i^2 =k_j^2\}$ contains at most $\frac{1}{2}N(N-1)$ elements.
\end{lemma}
\begin{proof}
  All $k_i$ are distinct since (by definition) all $\mu_i$ are distinct. So for $ i\neq j$, $k_i^2 = k_j^2 \Rightarrow \lambda=-\frac{1}{2}(\mu_i+\mu_j)\in\mathcal{S}$. The number of (unique) elements in this set is at most the number of unique pairs $(i,j), i,j\leq N, i\neq j$, which equals $\frac{1}{2}N(N-1)$.
\end{proof}
\begin{lemma} \label{lmm:ODEEquivalence}
  Let $\lambda\notin\mathcal{S}$. Then there exist unique vectors $\zeta = [\zeta_0,\ldots,\zeta_{N-1}] \in \CC^N$ and $\beta = [\beta_0,\ldots,\beta_N] \in \CC^{N + 1}$, depending on $\lambda$ and such that for every $q \in \CONTD{2N}(\BOmega)$ one has
  \small
  \begin{displaymath}    
    (\zeta_0 + \zeta_1 D_x^2 + \ldots + \zeta_{N-1} D_x^{2N-2} + D_x^{2N})\Delta(\lambda) q = (\beta_0 + \beta_1 D_x^2 + \ldots + \beta_{N-1} D_x^{2N-2} + \beta_{N} D_x^{2N})q 
  \end{displaymath}
  \normalsize
\end{lemma}
\begin{proof}
  Let $Q \DEF [q,q^{(2)},\ldots,q^{(2N)}]$. Repeated differentiation of \eqref{SecondDerivativeIdentity} yields the following system of equations:
  \begin{equation}
    \label{DerivativeSystem}
    \left[
      \begin{array}{r} 
        \Delta(\lambda) q \\ D_x^2\Delta(\lambda) q \\ 
        \multicolumn{1}{c}{\vdots}
        \\ D_x^{2N}\Delta(\lambda) q 
      \end{array} 
    \right] =  M Q - V
  \end{equation}
  where
  \begin{displaymath}
    M \DEF e^{\lambda\tau_0}(\lambda + \alpha)I + 2
    \underbrace{
    \begin{bmatrix}
      0                  & 0           & 0            & \hdots      & 0\\
      (c\cdot k)         & 0           & 0            & \hdots      & 0\\
      (c\cdot k^3)       & (c\cdot k)  & 0            & \hdots      & 0\\
      \vdots             & \ddots      & \ddots       & \ddots      & \vdots\\
      (c \cdot k^{2N-1}) & \hdots      & (c\cdot k^3) & (c\cdot k)  & 0\\
    \end{bmatrix}}_{\DEF \Xi^T}, \qquad
    V \DEF 
    \left[
      \begin{array}{r}
        (c\cdot Kq) \\ 
        (ck^2\cdot Kq) \\ 
        (ck^4\cdot Kq) \\ 
        \multicolumn{1}{c}{\vdots}\\
        (ck^{2N}\cdot Kq) 
      \end{array}
    \right]
  \end{displaymath}
  and $I$ is the identity matrix of size $N + 1$. (Note that \eqref{DerivativeSystem} is an equality that holds on $\BOmega$. Also, the definition of $\Xi^T$ is not used in the current proof, but will reoccur in Appendix \ref{app:ProofCharPolynom}.) We take a linear combination of the rows in \eqref{DerivativeSystem} with the components of the vector $Z \DEF [\zeta, 1] \in \CC^{N+1}$ such that this combination of the elements of $V$ in (\ref{DerivativeSystem}) vanishes. This eliminates all integral terms entering \eqref{DerivativeSystem} via $Kq$. Thus we seek $\zeta$ such that $Z^T V = 0$, i.e.
  \begin{equation}
    \label{eq:spectrum:14}
    \begin{bmatrix}
      \zeta& 1
    \end{bmatrix}
    \underbrace{
      \begin{bmatrix}
        1 & 1 & \hdots & 1 \\
        k_1^2 & k_2^2 & \hdots & k_N^2 \\
        k_1^4 & k_2^4 & \hdots & k_N^4 \\
        \vdots & \vdots & & \vdots \\
        k_1^{2N} & k_2^{2N} & \hdots & k_N^{2N} 
      \end{bmatrix}}_{\DEF \hat{W}^T}
    \begin{bmatrix} 
      c_1K_1 q \\ c_2K_2 q \\ \vdots \\ c_NK_N q 
    \end{bmatrix} = 0
  \end{equation}
  on $\BOmega$. If this equation is to be satisfied for any $q$, then we must have $\hat{W}Z=0$, which is equivalent to
  \begin{equation} 
    \label{VandermondeDefinition}
    \underbrace{
      \begin{bmatrix}
        1 & k_1^2 & k_1^4 & \hdots & k_1^{2N-2} \\
        1 & k_2^2 & k_2^4 & \hdots & k_2^{2N-2} \\
        \vdots & \vdots & \vdots &  & \vdots \\
        1 & k_N^2 & k_N^4 & \hdots & k_N^{2N-2}
      \end{bmatrix}}_{\DEF W} 
    \zeta = -
    \begin{bmatrix} 
      k_1^{2N} \\ k_2^{2N} \\ \vdots \\ k_N^{2N} 
    \end{bmatrix}
\end{equation}
The $N \times N$ Vandermonde matrix $W$ is invertible 
 since all $k_i^2$ are distinct by Proposition \ref{prp:UniqueK}. Hence $\zeta$ can be found by applying $W^{-1}$ to \eqref{VandermondeDefinition}. To find $\beta$ we apply the row vector $[\zeta,1]^T$ from the left to \eqref{DerivativeSystem} to infer that $\beta^T = [\zeta,1]^T M$. Hence
\begin{equation} 
  \label{VandermondeInverse}
  \beta = M^T 
  \begin{bmatrix}
    \zeta\\ 1
  \end{bmatrix}
  = -M^T 
  \begin{bmatrix} 
    W^{-1} & \EMPTY \\ 
    \EMPTY & 1 
  \end{bmatrix}
  \begin{bmatrix} 
    k_1^{2N} \\ k_2^{2N} \\ \vdots \\ k_{N}^{2N} \\ -1 
  \end{bmatrix}
\end{equation}
which concludes the proof.
\end{proof}

\begin{remark}\label{rmk:Order2N}
  For $\lambda\in\mathcal{S}$ the vectors $\zeta$ and $\beta$ still exist, but they are not unique, as can be seen from \eqref{VandermondeDefinition}. For simplicity we do not consider this case here. \hfill \QEDEX
\end{remark}

\begin{theorem}\label{thm:ODE}
  Suppose $\lambda \not\in \mathcal{S}$ and let $\{\beta_i\}_{i=1}^N$ as in Lemma \ref{lmm:ODEEquivalence}. Then $\Delta(\lambda)q = 0$ implies
  \begin{equation}
    \label{ODE01}
    \beta_0 q + \beta_1 q^{(2)} + \ldots + \beta_{N-1} q^{(2N-2)} + \beta_N q^{(2N)} = 0 
\end{equation}
\end{theorem}
\begin{proof}
  Since $\Delta(\lambda) q=0$ on $\BOmega$ it holds that $D^m\Delta(\lambda) q=0$ for all $m \in \NN$. The result now follows from Lemma \ref{lmm:ODEEquivalence}.
\end{proof}
Our next objective is to obtain what one could call a converse to the above theorem. Specifically, we ask when for a given $\lambda \in \CC$ with $\lambda \not\in \mathcal{S}$ a solution $q$ of \eqref{ODE01} also satisfies $\Delta(\lambda)q=0$. For this we start by noting that eigenvalues of the ODE \eqref{ODE01} are roots of the characteristic polynomial
\begin{equation} 
  \label{ODE01CharPolynom} 
  \mathcal{P}(\rho) = \beta_{N} \rho^{2N} + \beta_{N-1}\rho^{2N-2} + \ldots + \beta_1\rho^2 + \beta_0
\end{equation}
Evaluating the coefficients $\beta_i$ of this polynomial by means of \eqref{VandermondeInverse} yields the following result. Its proof may be found in Appendix \ref{app:ProofCharPolynom}.
\begin{proposition}\label{prp:ODE01CharEq}
For $\lambda \not\in \mathcal{S}$ the characteristic polynomial $\mathcal{P}$ is given by
\begin{equation}\label{ODE01CharPolynom2}
  \mathcal{P}(\rho) = \frac{e^{\lambda\tau_0}(\lambda+\alpha)}{2}\prod_{j=1}^N (\rho^2-k_j(\lambda)^2) + \sum_{i=1}^{N}{c_i k_i(\lambda) \prod_{\substack{j=1 \\ j\neq i}}^N{(\rho^2-k_j(\lambda)^2)}}
\end{equation}
\end{proposition}
Since $\mathcal{P}$ is an even function, it follows that if $\rho \in \CC$ is an eigenvalue of \eqref{ODE01} then the same is true for $-\rho$. 
\begin{proposition}\label{prp:EigenFunctionAnsatz}
  If \eqref{ODE01CharPolynom2} has $2N$ distinct roots ${\pm\rho_1(\lambda),\ldots,\pm\rho_N(\lambda)}$ then the general solution of \eqref{ODE01} is of the form
  \begin{equation}
    \label{EigenFunctionAnsatz}
    q_{\lambda}(x) = \sum_{i=1}^{N}{\bigl[\gamma_i e^{\rho_i(\lambda) x} + \gamma_{-i} e^{-\rho_i(\lambda) x}\bigr]}  \qquad \forall\,x \in \BOmega
  \end{equation}
  where the coefficients $\gamma_{\pm i} \in \CC$ are arbitrary.
\end{proposition}
For \eqref{EigenFunctionAnsatz} to satisfy $\Delta(\lambda)q_{\lambda}=0$, from \eqref{eq:spectrum:11} we see that
\begin{equation}
  \label{eq:spectrum:16}
  \begin{aligned}
    0 = (\Delta(\lambda)q)(x) &= e^{\lambda \tau_0}(\lambda+\alpha) \sum_{i=1}^{N}{\bigl[\gamma_i e^{\rho_i x} + \gamma_{-i} e^{-\rho_i x}\bigr]}\\
    &- \sum_{j=1}^{N}{c_j\sum_{i=1}^N{\Bigl[\gamma_i\int_{-1}^1{e^{-k_j|x-r|+\rho_i r}\,dr} + \gamma_{-i}\int_{-1}^1{e^{-k_j|x-r|-\rho_i r}\,dr}\Bigr]}}
  \end{aligned}
\end{equation}
must hold for all $x \in \BOmega$. For notational convenience we have suppressed the dependence on $\lambda$ of $q$, $\rho_i$ and $k$. Recalling that $\BOmega = [-1,x] \cup [x,1]$ for each fixed $x \in \BOmega$, we split the domains of integration accordingly. If 
\begin{equation}
  \label{eq:spectrum:15}
  k_j(\lambda) \neq \pm\rho_i(\lambda) \qquad \forall\,i,j=1,2,\ldots,N  
\end{equation}
then \eqref{eq:spectrum:16} becomes
\begin{align*}
  0 &= e^{\lambda \tau_0}(\lambda+\alpha) \sum_{i=1}^{N}{\bigl[\gamma_i e^{\rho_i x} + \gamma_{-i} e^{-\rho_i x}\bigr]}\\
  &- \sum_{j=1}^{N}{c_j\sum_{i=1}^N{\gamma_i \Bigl[\frac{2k_j}{k_j^2-\rho_i^2}e^{\rho_i x} - \frac{e^{-(k_j - \rho_i)}}{k_j-\rho_i}e^{k_j x} - \frac{e^{-(k_j+\rho_i)}}{k_j + \rho_i}e^{-k_j x}\Bigr]}}\\
  &+ \sum_{j=1}^N{c_j\sum_{i=1}^N{\gamma_{-i} \Bigl[\frac{2k_j}{k_j^2-\rho_i^2}e^{-\rho_i x} - \frac{e^{-(k_j + \rho_i)}}{k_j+\rho_i}e^{k_j x} - \frac{e^{-(k_j-\rho_i)}}{k_j - \rho_i}e^{-k_j x} \Bigr]}} 
\end{align*}
Sorting the terms according to their exponents in $x$ while again suppressing dependence on $\lambda$ of $\rho_i$ and $k$ yields
\small
\begin{align*}
  0 &= \sum_{i=1}^{N}{\Bigl\{ \gamma_i e^{\rho_i x} \Bigl[e^{\lambda\tau_0}(\lambda+\alpha) - \sum_{j=1}^N{\frac{2c_j k_j}{k_j^2-\rho_i^2}} \Bigr] + \gamma_{-i} e^{-\rho_i x} \Bigl[e^{\lambda\tau_0}(\lambda+\alpha) -  \sum_{j=1}^N{\frac{2 c_jk_j}{k_j^2-\rho_i^2}} \Bigr]\Bigr\}}\\
  &+ \sum_{j=1}^{N} c_j e^{-k_j}  \Bigl\{e^{k_j x}\Bigl[\sum_{i=1}^{N} \gamma_i \frac{e^{\rho_i}}{k_j-\rho_i} + \sum_{i=1}^{n} \gamma_{-i} \frac{e^{-\rho_i}}{k_j+\rho_i} \Bigr] + e^{-k_j x} \Bigl[\sum_{i=1}^{N} \gamma_i \frac{e^{-\rho_i}}{k_j + \rho_i} + \sum_{i=1}^{n} \gamma_{-i} \frac{e^{\rho_i}}{k_j-\rho_i} \Bigr]\Bigr\}
\end{align*}
\normalsize
Proposition \ref{prp:ODE01CharEq} guarantees that all coefficients of $e^{\pm\rho_i(\lambda)x}$ vanish. As for the remaining terms, all coefficients of $e^{\pm k_j(\lambda)x}$ should vanish as well. Thus we must have
\begin{align}
  \begin{split}
    \sum_{j=1}^{N}{c_j e^{-k_j} e^{k_j x} \Bigl[\sum_{i=1}^{N}{\gamma_i \frac{e^{\rho_i}}{k_j-\rho_i}} + \sum_{i=1}^{n}{\gamma_{-i} \frac{e^{-\rho_i}}{k_j+\rho_i}}\Bigr]} &= 0 \\
    \sum_{j=1}^{N}{c_j e^{-k_j} e^{-k_j x} \Bigl[\sum_{i=1}^{N}{\gamma_i \frac{e^{-\rho_i}}{k_j + \rho_i}} + \sum_{i=1}^{n}{\gamma_{-i} \frac{e^{\rho_i}}{k_j-\rho_i}}\Bigr]} &= 0 
  \end{split} \nonumber
\end{align}
where dependence on $\lambda$ of $\rho_i$ and $k$ was suppressed. This yields a set of $2N$ linear equations: one for each $e^{\pm k_j(\lambda)x}$. With $\Gamma = [\gamma_1, \gamma_2,\ldots,\gamma_N, \gamma_{-1}, \gamma_{-2},\ldots,\gamma_{-N}]$ and the matrix $S(\lambda)$ defined by
\begin{equation}
  \label{CharacteristicMatrix}
  S(\lambda) \DEF 
  \begin{bmatrix}
    S^-_{\lambda} & S^+_{\lambda} \\
    S^+_{\lambda} & S^-_{\lambda}
  \end{bmatrix}
\end{equation} 
where
\begin{displaymath}
  [S^-_{\lambda}]_{j,i} \DEF \frac{e^{\rho_i(\lambda)}}{\lambda + \mu_j - \rho_i(\lambda)}, \qquad [S^+_{\lambda}]_{j,i} \DEF \frac{e^{-\rho_i(\lambda)}}{\lambda + \mu_j + \rho_i(\lambda)}
\end{displaymath}
we seek $\Gamma$ such that 
\begin{equation}
  \label{eq:spectrum:12}
  S(\lambda)\Gamma=0
\end{equation}
In order for this system to have a non-trivial solution $\Gamma = \Gamma_{\lambda}$, it is necessary (and sufficient) for the determinant of $S(\lambda)$ to vanish,
\begin{equation}
  \label{CharacteristicEquation}
  \DET{S(\lambda)} = 0 
\end{equation}
This result is summarised in the following theorem. 
\begin{theorem}\label{thm:CharacteristicEquation}
  Suppose that $\lambda \not\in \mathcal{S}$ and assume that the characteristic polynomial $\mathcal{P}$ in \eqref{ODE01CharPolynom2} has $2N$ distinct roots, denoted by $\pm\rho_{i}(\lambda)$ for $i=1,2,\ldots,N$. If $\lambda$ satisfies \eqref{CharacteristicEquation} and \eqref{eq:spectrum:15} then $\lambda \in \PS(A)$. The corresponding eigenfunction is $\varepsilon_\lambda\otimes q_{\lambda}$, with $q_{\lambda}$ given by \eqref{EigenFunctionAnsatz} with $\Gamma_{\lambda}$ a solution of \eqref{eq:spectrum:12}.
\end{theorem}
\begin{remark}
  \label{rem:spectrum:3}
  Two comments on the above Theorem seem in order.
  \begin{enumerate}[itemsep=1ex,leftmargin=2em]
  \item[(i)]
    The above procedure can easily be adapted to cover the degenerate cases excluded in Theorem \ref{thm:CharacteristicEquation}. All we need is to adjust the form of $q_{\lambda}$ in Proposition \ref{prp:EigenFunctionAnsatz}. We do not pursue this for reasons of clarity and readability. Rather, in specific instances we check that degeneracy is not an issue.
  \item[(ii)]
    We expect that that the order of $\lambda$ as a root of \eqref{CharacteristicEquation} equals the multiplicity of $\lambda$ as a pole of $R(\cdot,A)$, see Corollary \ref{cor:spectrum:3} in \S \ref{sec:spectrum:generalities}. This would give an explicit way to verify simplicity of critical eigenvalues in \S \ref{sec:cm}. We intend to comment on this issue in future work. \hfill \QEDEX
  \end{enumerate}
\end{remark}

\subsection{Resolvent}\label{sec:spectrum:resolvent}
Now that we are able to reduce determining the point spectrum, in this specific example and modulo a technical restriction, to a finite dimensional matrix problem, the next step is to determine the solution of the resolvent problem, 
\begin{equation}
  \label{eq:spectrum:17}
  (z - A)\psi=\phi
\end{equation}
i.e. to find a representation of $\phi \in X$ in terms of the given function $\psi \in X$ when $z \in \rho(A)$. For this task we see from Proposition \ref{prop:spectrum:1} in \S \ref{sec:spectrum:generalities} that we first need to solve
\begin{equation}
  \label{eq:spectrum:22}
  \Delta(z)\psi(0)= S_z \phi
\end{equation}
For our specific example the above is equivalent to an integral equation for $q \DEF \psi(0)$,
\begin{equation} 
  \label{ResolventDefinition}
  (z+\alpha)q(x) - \int_{-1}^{1}{J(x,r)e^{-z\tau_0-z|x-r|}q(r)\,dr} = h_{z}(x) \qquad \forall\,x \in \BOmega
\end{equation}
where
\begin{equation} 
  \label{ResolventHDef}
  h_{z}(x) \DEF \phi(0,x) + \int_{-1}^{1}{\int_{-\tau_0-|x-r|}^{0}{J(x,r)e^{-z(\tau_0+s)-z|x-r|} \phi(s,r)\,ds}\,dr}
\end{equation}
for all $x \in \BOmega$. Inspired by \eqref{EigenFunctionAnsatz} we propose the following variation-of-constants Ansatz for its solution
\begin{displaymath}
  q(x) = g(x) + \sum_{i=1}^{N}{\bigl[\gamma_i(x)e^{\rho_i x} + \gamma_{-i}(x)e^{-\rho_i x}\bigr]} \qquad \forall\,x \in \BOmega
\end{displaymath}
where $\rho_{\pm i}(z)$ are distinct roots of \eqref{ODE01CharPolynom}. We seek $g \in \CONT(\BOmega)$ and $\gamma_{\pm 1}, \ldots,\gamma_{\pm N} \in \CONTD{1}(\BOmega)$. Substitution into \eqref{ResolventDefinition} and suppressing dependence on $z$ of $h$, $\rho_i$ and $k$ yields
\small
\begin{align*}
  e^{z\tau_0}h(x)=& e^{z\tau_0}(z+\alpha)g(x) + e^{z\tau_0}(z+\alpha)\sum_{i=1}^{N}{\bigl[\gamma_i(x)e^{\rho_i x} + \gamma_{-i}(x)e^{-\rho_i x}\bigr]}\\
  &-\sum_{j=1}^{N}{c_j e^{k_j x} \Bigl\{\int_{x}^{1}{e^{-k_j r} g(r)\,dr} + \sum_{i=1}^{N}{\int_{x}^{1}{\bigl[\gamma_i(r)e^{(-k_j+\rho_i) r} + \gamma_{-i}(r)e^{(-k_j-\rho_i) r}\bigr]\,dr}} \Bigr\}}\\
  &-\sum_{j=1}^{N}{c_j e^{-k_j x} \Bigl\{\int_{-1}^{x}{e^{k_j r} g(r)\,dr} + \sum_{i=1}^{N}{\int_{-1}^{x}{\bigl[\gamma_i(r)e^{(k_j+\rho_i) r} + \gamma_{-i}(r)e^{(k_j-\rho_i) r}\bigr]}\,dr} \Bigr\}}
\end{align*}
\normalsize
If \eqref{eq:spectrum:15} holds, we may integrate by parts and rearrange the terms,
\small
\begin{equation}
  \label{ResolventSubstituted01}
  \begin{aligned}
    e^{z\tau_0}h(x)=& e^{z\tau_0}(z+\alpha)g(x) + e^{z\tau_0}(z+\alpha)\sum_{i=1}^{N}{\bigl[\gamma_i(x)e^{\rho_i x} + \gamma_{-i}(x)e^{-\rho_i x} \bigr]}\\
    &-\sum_{i,j=1}^{N} c_j \Bigl[\frac{e^{\rho_i x}}{k_j+\rho_i}\gamma_i(x) + \frac{e^{-\rho_i x}}{k_j-\rho_i}\gamma_{-i}(x) + \frac{e^{\rho_i x}}{k_j-\rho_i}\gamma_i(x) + \frac{e^{-\rho_i x}}{k_j+\rho_i}\gamma_{-i}(x)\Bigr]\\
    &+\sum_{j=1}^{N} c_j e^{k_j x} \Bigl\{ \sum_{i=1}^{N} \Bigl[ \frac{e^{-k_j+\rho_i}\gamma_i(1)}{k_j-\rho_i} + \frac{e^{-k_j-\rho_i}\gamma_{-i}(1)}{k_j+\rho_i} \Bigr] \\
    &\hphantom{+\sum_{j=1}^{N} c_j e^{k_j x} \Bigl\{} - \int_{x}^{1} e^{-k_j r} \Bigl[g(r) + \sum_{i=1}^{N} \frac{e^{\rho_i r}}{k_j-\rho_i}\gamma'_i(r) + \frac{e^{-\rho_i r}}{k_j+\rho_i}\gamma'_{-i}(r)\Bigr] dr \Bigr\}\\
    &+\sum_{j=1}^{N} c_j e^{-k_j x} \Bigl\{ \sum_{i=1}^{N} \Bigl[ \frac{e^{-k_j-\rho_i}\gamma_i(-1)}{k_j+\rho_i} + \frac{e^{-k_j+\rho_i}\gamma_{-i}(-1)}{k_j-\rho_i} \Bigr] \\
    &\hphantom{+\sum_{j=1}^{N} c_j e^{-k_j x} \Bigl\{} + \int_{-1}^{x} e^{k_j r} \Bigl[-g(r) + \sum_{i=1}^{N} \frac{e^{\rho_i r}}{k_j+\rho_i}\gamma'_i(r) + \frac{e^{-\rho_i r}}{k_j-\rho_i}\gamma'_{-i}(r)\Bigr] dr \Bigr\}
  \end{aligned}
\end{equation}
\normalsize
where again dependency of $h$, $\rho_i$ and $k$ on $z$ was suppressed. When $z \not\in \mathcal{S}$, Proposition \ref{prp:ODE01CharEq} is applied and all terms involving $e^{\pm\rho_i(z) x}$ drop out. We can choose $g = g_{z}$ as
\begin{displaymath}
  g_{z}(x) \DEF \frac{h_{z}(x)}{z+\alpha} \qquad \forall\,x \in \BOmega
\end{displaymath}
provided we can achieve that the remaining terms (i.e. the last four lines) of \eqref{ResolventSubstituted01} vanish. So for $j=1,2,\ldots,N$ it should hold that for every $x \in \BOmega$,
\footnotesize
\begin{align}
  \int_{x}^{1}{e^{-k_j r}\Bigl\{g(r) + \sum_{i=1}^{N}{\Bigl[\frac{e^{\rho_i r}}{k_j-\rho_i}\gamma'_i(r) + \frac{e^{-\rho_i r}}{k_j+\rho_i}\gamma'_{-i}(r)\Bigr]}\Bigr\}\,dr} - e^{-k_j}\sum_{i=1}^{N}{\Bigl[\frac{e^{\rho_i}\gamma_i(1)}{k_j-\rho_i} + \frac{e^{-\rho_i}\gamma_{-i}(1)}{k_j+\rho_i}\Bigr]} &= 0 \nonumber \\ 
  \int_{-1}^{x}{e^{k_j r} \Bigl\{-g(r) + \sum_{i=1}^{N}{\Bigl[\frac{e^{\rho_i r}}{k_j+\rho_i}\gamma'_i(r) + \frac{e^{-\rho_i r}}{k_j-\rho_i}\gamma'_{-i}(r)\Bigl]}\Bigr\}\,dr} + e^{-k_j}\sum_{i=1}^{N}{\Bigl[\frac{e^{-\rho_i}\gamma_i(-1)}{k_j+\rho_i} + \frac{e^{\rho_i}\gamma_{-i}(-1)}{k_j-\rho_i}\Bigr]} &= 0 \label{ResolventSystem01} 
\end{align}
\normalsize
with the same notational convention as before. We seek functions $\gamma_{\pm i}$ such that the integrands and the remaining terms in \eqref{ResolventSystem01} vanish. This yields the system
\begin{displaymath}
  \underbrace{
    \begin{bmatrix} 
      T^-_z & T^+_z\\ 
      T^+_z & T^-_z 
    \end{bmatrix}}_{\DEF T(z)} 
  \begin{bmatrix} 
    P^+_{z}(x) & \EMPTY\\ 
    \EMPTY & P^-_{z}(x) 
  \end{bmatrix} 
  \Gamma'(x) = \frac{h_{z}(x)}{z+\alpha} 
  \left[
  \begin{array}{r} 
    -\mathbf{1}\\ 
    \mathbf{1} 
  \end{array}
  \right] \qquad \forall\,x \in \BOmega
\end{displaymath}
where $\mathbf{1} \in \RR^N$ is the vector with one on each entry,
\begin{equation}
  \label{eq:spectrum:18}
  [T^{\pm}_{z}]_{j,i} \DEF \frac{1}{k_j(z) \pm \rho_i(z)}, 
  \qquad P^{\pm}_{z}(x) \DEF 
  \text{diag}_N\bigl[e^{\pm\rho_1(z) x},\ldots,e^{\pm\rho_N(z) x}\bigr]
\end{equation}
and 
\begin{displaymath}
  \Gamma = [\gamma_1, \ldots, \gamma_N, \gamma_{-1},\ldots,\gamma_{-N}]
\end{displaymath}
If the matrix $T(z)$ is invertible, we find $\Gamma = \Gamma_{z}$ by matrix inversion and integration,
\begin{equation} \label{GammaDefinition01}
  \Gamma_{z}(x) = \Gamma_{0,z} + \underbrace{\int_{x_0}^x{\frac{h_{z}(r)}{z+\alpha} 
      \begin{bmatrix} 
        P^-_{z}(r) & \EMPTY\\ 
        \EMPTY & P^+_{z}(r) 
      \end{bmatrix} 
      T(z)^{-1}
      \left[
      \begin{array}{r} 
        -\mathbf{1} \\ \mathbf{1} 
      \end{array}
      \right]\,dr}}_{\DEF \hat{\Gamma}_{z}(x)}
\end{equation}
for some initial reference point $x_0$ in $\BOmega$ and integration constants $\Gamma_{0,z}\in \CC^{2N}$. Any choice of integration constants results in a choince for $\Gamma_z$ for which the integral terms in \eqref{ResolventSystem01} vanish. In order to satisfy the remaining terms in \eqref{ResolventSystem01}, $\Gamma_{0,z}$ is chosen as
\begin{equation} \label{Gamma0Definition}
  \Gamma_{0,z} = - S(z)^{-1} 
  \begin{bmatrix}
    S^-_{z} & S^+_{z} & 0 & 0 \\
    0 & 0 & S^+_{z} & S^-_{z}
  \end{bmatrix}
  \left[
    \begin{array}{l} 
      \hat{\Gamma}_{z}(1) \\ \hat{\Gamma}_{z}(-1) 
    \end{array}
  \right]
\end{equation}
for $S(z)$, $S^+_{z}$, and $S^-_{z}$ as in \eqref{CharacteristicMatrix}. Clearly, $S(z)^{-1}$ exists if and only if $\DET{S(z)} \neq0$, which is consistent with the fact that the resolvent operator $R(z,A)$ is not defined when $z \in \PS(A)$. We are now ready to formulate the key result of this section.
\begin{theorem}
  \label{thm:spectrum:1}
  Suppose that $z \in \rho(A)$ and
  \begin{itemize}
  \item
    $z \not\in \mathcal{S}$;
  \item the characteristic polynomial $\mathcal{P}$ has $2N$ distinct roots;
  \item condition \eqref{eq:spectrum:15} holds and;
  \item the matrix $T(z)$ is invertible.
  \end{itemize}
  Then the solution of \eqref{eq:spectrum:17} is given by $\psi_{z} = \varepsilon_z \otimes q_{z} + H_z \phi$ with
  \begin{equation}
    \label{eq:spectrum:23}
    q_{z}(x) = \frac{h_{z}(x)}{z+\alpha} + \sum_{i=1}^{N}{\bigl[\gamma_{i,z}(x)e^{\rho_i(z) x} + \gamma_{-i,z}(x)e^{-\rho_i(z) x}\bigr]} \qquad \forall\,x \in \BOmega
  \end{equation}
  with $\Gamma_{z}$ given by \eqref{GammaDefinition01} and $h_{z}$ is as in \eqref{ResolventHDef}.
\end{theorem}

\begin{remark}\label{rem:spectrum:4}
The fourth condition in the above Theorem seems peculiar and of a different nature than the first three, which already occured as simplifying conditions in \S \ref{sec:spectrum:chareq}. We refrain from investigating this issue here. Wherever we need the result of this theorem, we check the fourth condition explicitly. \hfill \QEDEX
\end{remark}


\section{Normal forms for local bifurcations}\label{sec:cm}
Let $\hat{\phi} \in X$ be a stationary point of the semiflow generated by \eqref{eq:neuralfield:8}. By Theorem \ref{thm:setting:2} in \S \ref{sec:setting:sunstar} the linearisation of this semiflow at $\hat{\phi}$ defines a strongly continuous semigroup $T$ of bounded linear operators on $X$, generated by $A$ as in \eqref{eq:setting:11} and \eqref{eq:spectrum:7}. If $F$ is as in \eqref{eq:setting:6} then $T$ will be the solution semigroup of the linear problem \eqref{eq:spectrum:4}, which is of the form \eqref{eq:spectrum:1}. In the present section we prepare for the computation of a critical normal form when $\hat{\phi}$ undergoes a Hopf or a double Hopf bifurcation. The actual computation is performed in \S \ref{sec:numerics}.
\par
In \S \ref{sec:setting:sunstar} we alluded to the fact that a reformulation of equations of type \eqref{eq:neuralfield:8}, such as \eqref{eq:neuralfield:10}, as an abstract integral equation of type \eqref{eq:setting:7} allows for a relatively straightforward application of basic dynamical results such as the center manifold theorem. Indeed, by \eqref{eq:setting:7} and the exponential estimates of Proposition \ref{prop:spectrum:2} in \S \ref{sec:spectrum:generalities} the general center manifold theory for AIE presented in \cite[Ch. IX]{Diekmann1995} is directly applicable to \eqref{eq:neuralfield:8} in the setting of \S \ref{sec:setting:sunstar}. We shall relegate a more detailed technical presentation to the forthcoming paper \cite{VanGils2012b}.
\par
There exists an efficient approach based on Fredholm's solvability condition towards the derivation of explicit formulas for critical normal form coefficients of local bifurcations of dynamical systems. Once such formulas have been derived for a certain class of dynamical systems, they may be evaluated for specific equations using spectral information from the linearisation at the critical equilibrium or fixed point, together with information on the higher order derivatives of the particular non-linearity. The technique goes back to \cite{Coullet1983} and has been successfully applied to ordinary differential equations \cite{Kuznetsov1999b}, \cite[\S 8.7]{Kuznetsov2004}) and iterated maps \cite{Meijer2006}, \cite{Kuznetsov2005b}, \cite{Govaerts2007}. The resulting formulas have been implemented in the software packages \verb#CONTENT# \cite{Kuznetsov1997}, its successor \verb#MATCONT# \cite{Dhooge2003} and \verb#CL_MATCONT# for maps. 
\par
In the forthcoming paper \cite{Janssens2012} the method is applied to AIE and DDE. Here we briefly summarise the results related to Hopf and double Hopf bifurcations, obtained using the Fredholm solvability technique, see Lemma \ref{lem:normalforms:1} below. In the Hopf case the corresponding formulae have been first obtained in \cite{VanGils1984,Diekmann1995} using a different method. As expected, the formulae given below look very similar to those given in \cite{Kuznetsov1999b} and \cite[\S 8.7]{Kuznetsov2004}. However, one should pay special attention to their proper interpretation in the current functional analytic context.

\subsection{Preliminaries}\label{sec:normalforms:prelims}
In \S\S \ref{sec:normalforms:hopf} and \ref{sec:normalforms:double-hopf} we will consider the situation that $\hat{\phi} \in X$ is a stationary point of the non-linear semiflow generated by \eqref{eq:neuralfield:8} and the linearised problem takes the form \eqref{eq:spectrum:1} with $A$ as in \eqref{eq:spectrum:7} and $L_z$ compact for all $z \neq -\alpha$. There is no loss of generality in assuming that $\hat{\phi} \equiv 0$. Suppose that $A$ has $n_c \ge 1$ simple eigenvalues on the imaginary axis, counting multiplicities.
\begin{remark}
  \label{rem:normalforms:2}
  One may show that $\sigma(A) = \sigma(\STAR{A}) = \sigma(\SUN{A}) = \sigma(\SUNSTAR{A})$, see \cite[p. 100 - 101]{Diekmann1995} and also \cite[Proposition IV.2.18]{Engel2000}. We will use this fact in the remainder of this section. For a detailed discussion of the `lifting' of the spectral properties of $A$ to corresponding properties of the various (adjoint) generators, we refer to \cite[p. 100 - 101]{Diekmann1995}. \hfill \QEDEX
\end{remark}
This implies the existence of a non-trivial center subspace $\XC$ of finite dimension $n_c$ and spanned by some basis $\Phi$ consisting of (generalized) eigenvectors corresponding to the critical eigenvalues of $A$. There exists a locally invariant center manifold $\LCM$ that is tangent to $\XC$ at the origin. One can show that on $\LCM$ the solution satisfies the abstract ODE 
\begin{displaymath}
  \dot{u}(t) = j^{-1}\left(\SUNSTAR{A}ju(t) + R(u(t))\right) \qquad \forall\,t \in \RR
\end{displaymath}
where the non-linearity $R$ is given by Lemma \ref{lem:setting:5} and is as smooth as the mapping $F$ appearing in \eqref{eq:setting:6}. Let $\xi(t)$ be the projection of $u(t)$ onto $\XC$. Then $\xi(t)$ can be expressed uniquely relatively to $\Phi$. The corresponding coordinate vector $z(t)$ of $\xi(t)$ satisfies some ODE that is smoothly equivalent to the normal form 
\begin{equation}
  \label{eq:bifurcation:37}
  \dot{z}(t) = \sum_{1 \le |\nu| \le 3}{g_{\nu}z^{\nu}(t)} + O(|z(t)|^{4}) \qquad \forall t \in \RR
\end{equation}
with unknown critical normal form coefficients $g_{\nu} \in \RR^{n_c}$. Here $\nu$ stands for a multi-index of length $n_c$. If $F$ is sufficiently smooth, we may define 
\begin{subequations}
  \begin{alignat}{6}
    B &\in \BND_2(X,\SUNSTAR{X}), \qquad &B(\phi_1,\phi_2)) &\DEF D^2R(0)(\phi_1,\phi_2)\label{eq:normalforms:3}\\
    C &\in \BND_3(X,\SUNSTAR{X}),  \qquad &C(\phi_1,\phi_2,\phi_3) &\DEF D^3R(0)(\phi_1,\phi_2,\phi_3) \label{eq:normalforms:4}
  \end{alignat}
\end{subequations}
for all $\phi_i \in X$. The nonlinearity $R : X \to \SUNSTAR{X}$ may then be expanded as
\begin{equation}
  \label{eq:bifurcation:39}
  R(\phi) = \frac{1}{2}B(\phi,\phi) + \frac{1}{3!}C(\phi,\phi,\phi) + O(\|\phi\|^4)
\end{equation}
Let $\CMMAPC : V \subset \RR^{n_c} \to X$ be a mapping that is as smooth as $F$ and defined on a neighbourhood $V$ of the origin in the coordinate space $\RR^{n_c}$ with image $\CMMAPC(V) = \LCM$. Then $\CMMAPC$ admits an expansion
\begin{equation}
  \label{eq:bifurcation:40}
  \CMMAPC(z) = \sum_{1 \le |\nu| \le 3}{\frac{1}{\nu!}h_{\nu}z^{\nu}} + O(|z|^4)
\end{equation}
where $\nu$ is a multi-index of length $n_c$ and $h_{\nu} \in X$ is an unknown coefficient. By the invariance of $\LCM$ we have
\begin{displaymath}
 \CMMAPC(z(t)) = u(t) \qquad \forall\,t \in \RR 
\end{displaymath}
Differentiating both sides with respect to time leads to the \emph{homological equation}
\begin{equation}
  \label{eq:bifurcation:41}
  \SUNSTAR{A}j\CMMAPC(z) + R(\CMMAPC(z)) = j(D\CMMAPC(z)\dot{z})
\end{equation}
Substituting the expansions \eqref{eq:bifurcation:37}, \eqref{eq:bifurcation:39} and \eqref{eq:bifurcation:40} into \eqref{eq:bifurcation:41} and equating coefficients of the corresponding powers of $z$, one recursively obtains the unknown coefficients $h_{\nu}$ and $g_{\nu}$ by solving linear operator equations of the form
\begin{equation}
  \label{eq:normalforms:1}
  (\lambda - \SUNSTAR{A})\SUNSTAR{\phi} = \SUNSTAR{\psi}
\end{equation}
where $\lambda \in \CC$ and $\SUNSTAR{\psi} \in \SUNSTAR{X}$ is given. If $\lambda \not\in \sigma(A)$ then \eqref{eq:normalforms:1} has a unique solution $\SUNSTAR{\phi} \in D(\SUNSTAR{A})$ for any given right-hand side. On the other hand, when $\lambda \in \sigma(A)$ a solution $\SUNSTAR{\phi}$ of \eqref{eq:normalforms:1} need not exist for all right-hand sides $\SUNSTAR{\psi}$. The following key lemma provides a condition for solvability that is useful in this situation.
\begin{lemma}[Fredholm solvability]
  \label{lem:normalforms:1}
  Let $\lambda \in \CC \setminus \{-\alpha\}$. Suppose that $L_{\lambda} \in \BND(Y)$ defined in \eqref{eq:spectrum:8} is compact. Then $\lambda - \SUN{A} : D(\SUN{A}) \subset \SUN{X} \to \SUN{X}$ has closed range. In particular, \eqref{eq:normalforms:1} is solvable for $\SUNSTAR{\phi} \in D(\SUNSTAR{A})$ given $\SUNSTAR{\psi} \in \SUNSTAR{X}$ if and only if $\PAIR{\SUN{\phi}}{\SUNSTAR{\psi}} = 0$ for all $\SUN{\phi} \in N(\lambda - \STAR{A})$.
\end{lemma}
\begin{proof}
  From Corollary \ref{cor:spectrum:4} in \S \ref{sec:spectrum:generalities} we infer that $\RG(\lambda - \STAR{A})$ is closed. We first prove that this implies that $\RG(\lambda - \SUN{A})$ is closed as well. Indeed, let $(\SUN{\psi}_n)_{n \in \NN}$ be a sequence in $\RG(\lambda - \SUN{A})$ such that $\SUN{\psi}_n \to \SUN{\psi} \in \SUN{X}$. Then there is a sequence $(\SUN{\phi}_n)_{n \in \NN}$ in $D(\SUN{A})$ such that
  \begin{displaymath}
    \SUN{\psi}_n = (\lambda - \SUN{A})\SUN{\phi}_n = (\lambda - \STAR{A})\SUN{\phi}_n  \qquad \forall\,n \in \NN
  \end{displaymath}
  where \eqref{eq:setting:10} was used in the second equality. Hence $\SUN{\psi}_n \in \RG(\lambda - \STAR{A})$ for all $n \in \NN$, so there exists $\SUN{\phi} \in D(\STAR{A})$ such that $(\lambda - \STAR{A})\SUN{\phi} = \SUN{\psi}$. Now 
  \begin{displaymath}
    \STAR{A}\SUN{\phi} = -(\lambda - \STAR{A})\SUN{\phi} + \lambda \SUN{\phi} = -\SUN{\psi} + \lambda \SUN{\phi} \in \SUN{X}
  \end{displaymath}
  so $\SUN{\phi} \in D(\SUN{A})$ and $(\lambda - \SUN{A})\SUN{\phi} = \SUN{\psi}$ by \eqref{eq:setting:10}. Hence $\SUN{\psi} \in \RG(\lambda - \SUN{A})$.
  \par
The second statement in the lemma is obtained from Banach's Closed Range Theorem \cite[\S VII.5]{Yosida1980} by which it follows that \eqref{eq:normalforms:1} has a solution if and only if $\SUNSTAR{\psi}$ annihilates $N(\lambda - \SUN{A})$, i.e. if and only if
  \begin{displaymath}
    \PAIR{\SUN{\phi}}{\SUNSTAR{\psi}} = 0 \qquad \forall\,\SUN{\phi} \in N(\lambda - \SUN{A})
  \end{displaymath}
  To conclude the proof we show that $N(\lambda - \SUN{A}) = N(\lambda - \STAR{A})$. Indeed, $N(\lambda - \SUN{A}) \subseteq N(\lambda - \STAR{A})$ by virtue of \eqref{eq:setting:10}. Conversely, suppose that $\SUN{\phi} \in N(\lambda - \STAR{A})$. Then $\SUN{\phi} \in D(\STAR{A})$ and $\STAR{A}\SUN{\phi} = \lambda \SUN{\phi} \in \SUN{X}$. Hence $N(\lambda - \SUN{A}) \supseteq N(\lambda - \STAR{A})$ again by \eqref{eq:setting:10}.
\end{proof}

\subsection{The Andronov-Hopf critical normal form}\label{sec:normalforms:hopf}
In this case $\sigma(A)$ contains a simple purely imaginary pair $\lambda_{1,2} = \pm i\omega_0$ with $\omega_0 > 0$ and no other eigenvalues on the imaginary axis. Let $\phi$ and $\SUN{\phi}$ be complex eigenvectors of $A$ and $\STAR{A}$ corresponding to $\lambda_1 = i\omega_0$ and satisfying $\PAIR{\phi}{\SUN{\phi}}=1$. The restriction of (\ref{eq:neuralfield:8}) to the critical center manifold $\LCM$ is smoothly equivalent to the  {\em Poincar\'{e} normal form}
\begin{equation}
  \label{Hopf.nf}
  \dot{z}=i \omega_0 z + g_{21} z|z|^2 + O(|z|^4)
\end{equation}
where $z$ is complex and the critical normal form coefficient $g_{21}$ is unknown. Any point $\xi$ in the \emph{real} two-dimensional center subspace $\XC$ corresponding to $\lambda_{1,2}$ may be uniquely expressed with respect to the set $\Phi=\{\phi,\BAR{\phi}\}$ by means of the smooth complex coordinate mapping 
\begin{displaymath}
  \xi \mapsto (z,\BAR{z}), \quad z := \PAIR{\xi}{\SUN{\phi}}
\end{displaymath}
The homological equation \eqref{eq:bifurcation:41} presently becomes
\begin{displaymath}
  \SUNSTAR{A}j\CMMAPC(z,\BAR{z}) + R(\CMMAPC(z,\BAR{z})) =
  j\left(D_z\CMMAPC(z,\BAR{z})\dot{z} + D_{\BAR{z}}\CMMAPC(z,\BAR{z})\dot{\BAR{z}}\right)
\end{displaymath}
with center manifold expansion
\begin{displaymath}
  \CMMAPC(z,\BAR{z}) = z\phi + \BAR{z}\BAR{\phi} + \sum_{2 \le j+k \le 3}{\frac{1}{j!k!}h_{jk}z^j\BAR{z}^k} + O(|z|^4)
\end{displaymath}
Note that since the image of $\CMMAPC$ lies in the real space $X$, it follows that its coefficients satisfy $h_{kj} = \BAR{h}_{jk}$. The derivates $\dot{z}$ and $\dot{\BAR{z}}$ are given by \eqref{Hopf.nf} and its complex conjugate.
\par
Comparing coefficients of the quadratic terms $z^2$ and $z\BAR{z}$ leads to two non-singular linear equations for $jh_{20}$ and $jh_{11}$ with solutions
\begin{equation}
  \label{eq:Hopf.h}
  \begin{aligned}
    jh_{20} &= -(\SUNSTAR{A})^{-1}B(\phi,\BAR{\phi})\\
    jh_{11} &= (2i\omega_0-\SUNSTAR{A})^{-1}B(\phi,\phi)
  \end{aligned}
\end{equation}
There are two equations corresponding to the cubic terms $z^3$ and $z^2\BAR{z}$, the first of which is non-singular. The second one reads
\begin{equation}\label{eq:bifurcation:51}
  (i\omega_0I - \SUNSTAR{A})jh_{21} = C(\phi,\phi,\BAR{\phi}) + B(\BAR{\phi},h_{20}) + 2B(\phi,h_{11}) - 2g_{21}j\phi
\end{equation}
An application of Lemma \ref{lem:normalforms:1} to \eqref{eq:bifurcation:51} yields
\begin{equation}
  \label{Hopf.c1}
  g_{21}=\frac{1}{2} \langle \SUN{\phi},C(\phi,\phi,\BAR{\phi}) + B(\BAR{\phi},h_{20}) + 2B(\phi,h_{11})\rangle
\end{equation}
with $h_{20}$ and $h_{11}$ implicitly given by \eqref{eq:Hopf.h}. The cubic coefficient $g_{21}$ determines the {\em first Lyapunov coefficient} $l_1$ by the formula 
\begin{displaymath}
  l_1=\frac{1}{\omega_0} \RE{g_{21}}  
\end{displaymath}
It is well known \cite{Kuznetsov2004} that in generic unfoldings of \eqref{Hopf.nf} $l_1 < 0$ implies a supercritical bifurcation of a limit cycle on the corresponding parameter-dependent locally invariant manifold, while $l_1>0$ implies a subcritical bifurcation of a limit cycle there.
\begin{remark}
  \label{rem:normalforms:3}
  Notice that the vector $\SUN{\phi}$ satisfies $\STAR{A}\SUN{\phi} = i\omega_0\SUN{\phi}$ instead of $\STAR{A}\SUN{\phi} = -i\omega_0\SUN{\phi}$ which is used in the finite dimensional case. The reason for this is that the pairing $\PAIR{\cdot}{\cdot}$ between $\SUN{X}$ and $\SUNSTAR{X}$ is complex-linear in \emph{both} arguments. Also, observe that the values of the multilinear form in \eqref{Hopf.c1} and \eqref{eq:Hopf.h} are elements of the dual space $\SUNSTAR{X}$ of $\SUN{X}$, i.e. they are (bounded) linear functionals. This has been taken into account in the numerical computations of \S \ref{sec:numerics}. A similar remark is valid for the expressions in \S \ref{sec:normalforms:double-hopf}. \hfill \QEDEX
\end{remark}

\subsection{The double Hopf critical normal form}\label{sec:normalforms:double-hopf}
In this case $\sigma(A)$ contains two simple purely imaginary pairs 
\begin{displaymath}
\lambda_{1,4} = \pm i\omega_1,~~\lambda_{2,3} = \pm i\omega_2  
\end{displaymath}
with $\omega_{1,2} > 0$, and no other eigenvalues on the imaginary axis. Let $\phi_{1,2}$ and $\SUN{\phi_{1,2}}$ be eigenvectors of $A$ and $\STAR{A}$,
\begin{displaymath}
  A\phi_1 = i\omega_1\phi_1, \quad A\phi_2 = i\omega_2\phi_2, \quad \STAR{A}\SUN{\phi_1} = i\omega_1\SUN{\phi_1}, \quad \STAR{A}\SUN{\phi_2} = i\omega_2\SUN{\phi_2}
\end{displaymath}
As in the finite-dimensional case, it is always possible to scale these vectors such that the `bi-orthogonality' relation
\begin{displaymath}
  \PAIR{\phi_j}{\SUN{\phi_i}} = \delta_{ij} \qquad (1 \le i,j \le 2)
\end{displaymath}
is satisfied. In addition, we assume the non-resonance conditions
\begin{equation}\label{eq:bifurcation:60}
  k\omega_1 \neq l\omega_2 \quad \mbox{for all } k,l \in \NN \mbox{ with } k + l \le 5
\end{equation}
Then the restriction of (\ref{eq:neuralfield:8}) to the critical center manifold $\LCM$ is smoothly equivalent to the  {\em Poincar\'{e} normal form}
\begin{equation}\label{NF.HH}
  \left\{
    \begin{aligned}
      \dot{z}_1 &= i\omega_1z_1 + g_{2100}z_1|z_1|^2 + g_{1011}z_1|z_2|^2 + g_{3200}z_1|z_1|^4 + g_{2111}z_1|z_1|^2|z_2|^2\\
      &\hphantom{{}=} + g_{1022}z_1|z_2|^4 + O(\|(z_1,\BAR{z_1},z_2,\BAR{z_2})\|^6)\\
      \dot{z}_2 &= i\omega_2z_2 + g_{1110}z_2|z_1|^2 + g_{0021}z_2|z_2|^2 + g_{2210}z_2|z_1|^4 + g_{1121}z_2|z_1|^2|z_2|^2\\
      &\hphantom{{}=} + g_{0032}z_2|z_2|^4 + O(\|(z_1,\BAR{z_1},z_2,\BAR{z_2})\|^6)
    \end{aligned}
  \right.
\end{equation}
where the constants $g_{jklm}$ are all complex \cite[Ch. VIII]{Kuznetsov2004}. Define
\begin{displaymath}
  \begin{bmatrix}
    p_{11}& p_{12}\\
    p_{21}& p_{22}
  \end{bmatrix}
  = 
  \begin{bmatrix}
    g_{2100}& g_{1011}\\
    g_{1110}& g_{0021}
  \end{bmatrix}
\end{displaymath}
and assume that 
\begin{displaymath}
  \RE{p_{11}}\RE{p_{12}}\RE{p_{21}}\RE{p_{22}} \neq 0
\end{displaymath}
As in shown in \cite[Ch. VIII]{Kuznetsov2004} the restriction of (\ref{NF.HH}) to $\LCM$  is locally \emph{smoothly orbitally equivalent} to 
\begin{equation}
  \label{eq:bifurcation:58}
  \left\{
    \begin{aligned}
      \dot{z}_1 &= i\omega_1z_1 + p_{11}z_1|z_1|^2 + p_{12}z_1|z_2|^2 + ir_1z_1|z_1|^4 + s_1z_1|z_2|^4\\
      &\hphantom{{}=} + O(\|(z_1,\BAR{z}_1,z_2,\BAR{z}_2)\|^6)\\
      \dot{z}_2 &= i\omega_2z_2 + p_{21}z_2|z_1|^2 + p_{22}z_2|z_2|^2 + s_2z_2|z_1|^4 + ir_2z_2|z_2|^4\\
      &\hphantom{{}=} + O(\|(z_1,\BAR{z}_1,z_2,\BAR{z}_2)\|^6)
    \end{aligned}
  \right.
\end{equation}
Here $p_{ij}$ and $s_{i}$ are complex while $r_i$ are real, for $1 \le i,j \le 2$. The real parts of $s_i$ are given by
\begin{displaymath}
  \RE{s_1} = \RE{g_{1022}} + \RE{g_{1011}} \times \left[\frac{\RE{g_{1121}}}{\RE{g_{1110}}} - 2\frac{\RE{g_{0032}}}{\RE{g_{0021}}} - \frac{\RE{g_{3200}}\RE{g_{0021}}}{\RE{g_{2100}}\RE{g_{1110}}}\right]
\end{displaymath}
and
\begin{displaymath}
  \RE{s_2} = \RE{g_{2210}} + \RE{g_{1110}} \times \left[\frac{\RE{g_{2111}}}{\RE{g_{1011}}} - 2\frac{\RE{g_{3200}}}{\RE{g_{2100}}} - \frac{\RE{g_{2100}}\RE{g_{0032}}}{\RE{g_{1011}}\RE{g_{0021}}}\right]  
\end{displaymath}
The real constants $r_i$ are of secondary importance in the bifurcation analysis of a generic two-parameter unfolding of \eqref{eq:bifurcation:58} and so we omit expressions for these. They can be extracted from the proof of \cite[Lemma 8.14]{Kuznetsov2004}.
\par
The double Hopf bifurcation is a complicated bifurcation, both from a computational as well as a conceptual viewpoint. An unfolding of \eqref{eq:bifurcation:58} is best analysed by rewriting it in polar coordinates. The sixth-order terms may not be truncated, since they may affect the qualitative dynamics. Depending on the sign of 
\begin{displaymath}
  \RE{p_{11}}\RE{p_{22}} = \RE{g_{2100}}\RE{g_{0021}} 
\end{displaymath}
this bifurcation exhibits either `simple' or `difficult' dynamics, see \cite[\S 8.6.2]{Kuznetsov2004}. Assuming generic dependence on parameters, one may encounter invariant tori, chaotic dynamics, Neimark-Sacker bifurcations of cycles and Shilnikov homoclinic orbits. Note that, although computations up to and including fifth order are required to determine \emph{all} critical coefficients, computations up to and including third order suffice to distinguish between `simple' and 'difficult' cases.
\par
The critical normal form coefficients may be obtained using a procedure similar to the Hopf case discussed in \S \ref{sec:normalforms:hopf}. We omit the details and only present the results, noting that the center manifold now has the formal expansion
\begin{displaymath}
  \CMMAPC(z_1,\BAR{z}_1,z_2,\BAR{z}_2) = z_1\phi_1 + \BAR{z}_1\BAR{\phi}_1 + z_2\phi_2 + \BAR{z}_2\BAR{\phi}_2 
+ \sum_{ j+k+l+m \geq 2}{\frac{1}{j!k!l!m!}h_{jklm}z_1^j\BAR{z}_1^kz_2^l\BAR{z}_2^m}
\end{displaymath}
At the second order in the corresponding homological equation, we find 
\begin{eqnarray*}
  jh_{1100}&=&-(\SUNSTAR{A})^{-1}B({\phi_{1}},{\overline{\phi}_{1}})\\
  jh_{2000}&=&(2i{\omega _{1}} - \SUNSTAR{A})^{-1}B({\phi_{1}},{\phi_{1}})\\
  jh_{1010}&=&[i(\omega _{1}+\omega _{2})- \SUNSTAR{A}]^{-1}B({\phi_{1}},{\phi_{2}})\\
  jh_{1001}&=&[i({\omega _{1}}-{\omega _{2}})-\SUNSTAR{A}]^{-1}B({\phi_{1}}, \,{\overline{\phi}_{2}})\\
  jh_{0020}&=&(2i{\omega _{2}}-\SUNSTAR{A})^{-1}B({\phi_{2}},{\phi_{2}})\\
  jh_{0011}&=&-(\SUNSTAR{A})^{-1}B({\phi_{2}},{\overline{\phi}_{2}})
\end{eqnarray*}
All operators in the right-hand side of the above equations are invertible due to the assumptions \eqref{eq:bifurcation:60} on the critical eigenvalues.
\par
Further, one obtains the following equations for $h_{jklm}$ with $j+k+l+m=3$:
\begin{eqnarray*}
  jh_{3000}&=&(3i\omega_1 - \SUNSTAR{A})^{-1}[C(\phi_1,\phi_1,\phi_1)+3B(h_{2000},\phi_1)]\\
  jh_{2010}&=&[i(2\omega_1+\omega_2) -\SUNSTAR{A}]^{-1} [C(\phi_1,\phi_1,\phi_2) + B(h_{2000},\phi_{2}) + 2B(h_{1010},\phi_1)]\\
  jh_{2001}&=&[i(2\omega_1-\omega_2)-\SUNSTAR{A}]^{-1} [C(\phi_1,\phi_1,\overline{\phi}_2) + B(h_{2000},\overline{\phi}_2) + 2B(h_{1001},\phi_1)]\\
  jh_{1020}&=&[i(\omega_1+2\omega_2)-\SUNSTAR{A}]^{-1} [C(\phi_1,\phi_2,\phi_2) + B(h_{0020},\phi_1) + 2B(h_{1010},\phi_2)]\\
  jh_{1002}&=&[i(\omega_1-2\omega_2)-\SUNSTAR{A}]^{-1} [C(\phi_1,\overline{\phi}_2,\overline{\phi}_2) + B(\overline{h}_{0020},\phi_1) + 2B(h_{1001},\overline{\phi}_2)]\label{E1002}\\
  jh_{0030}&=&(3i\omega_2 -\SUNSTAR{A})^{-1}[C(\phi_2,\phi_2,\phi_2) + 3B(h_{0020},\phi_2)]
\end{eqnarray*}
The cubic coefficients in the normal form (\ref{NF.HH}) come from the Fredholm solvability conditions and are given by
\begin{eqnarray*}
  g_{2100}&=&\frac{1}{2}\langle \SUN{\phi_1},C(\phi_1,\phi_1,\overline{\phi}_1) + B(h_{2000},\overline{\phi}_1) + 2B(h_{1100},\phi_1)\rangle\\  
  g_{1011}&=&\langle \SUN{\phi_1}, C(\phi_1,\phi_2,\overline{\phi}_2) + B(h_{1010},\overline{\phi}_2) + B(h_{1001},\phi_2) + B(h_{0011},\phi_1)\rangle\\   
  g_{1110}&=&\langle \SUN{\phi_2},C(\phi_1,\overline{\phi}_1,\phi_2) + B(h_{1100},\phi_2) + B(h_{1010},\overline{\phi}_1) + B(\overline{h}_{1001},\phi_1) \rangle\\ 
  g_{0021}&=&\frac{1}{2}\langle \SUN{\phi_2}, C(\phi_2,\phi_2,\overline{\phi}_2) + B(h_{0020},\overline{\phi}_2) + 2B(h_{0011},\phi_2) \rangle
\end{eqnarray*}
Similarly, one can compute all remaining coefficients in (\ref{NF.HH}) by proceeding to orders four and five. The resulting (lengthy) formulas are omitted. For the finite-dimensional case these can be found in \cite{Kuznetsov1999b}.

\subsection{Evaluation of normal form coefficients}\label{sec:normalforms:calc}
The computability of the normal form coefficients derived in the previous subsections depends on the possibility to evaluate the dual pairing $\PAIR{\SUN{\phi}}{\SUNSTAR{\phi}}$, where $\SUN{\phi} \in \SUN{X}$ is some eigenvector of $\STAR{A}$ corresponding to a simple eigenvalue $\lambda \in \sigma(A)$ and $\SUNSTAR{\phi} \in \SUNSTAR{X}$. Moreover, the coefficients $h_{\nu}$, with $\nu$ a certain multi-index, can only be computed once a representation for the resolvent $R(\lambda,\SUNSTAR{A})$ is known, where $\lambda \in \rho(A)$. At first sight this seems to be a difficult task, since $\SUN{X} = \STAR{Y} \times \LP{1}([0,h];\STAR{Y})$ and hence
\begin{displaymath}
  \SUNSTAR{X} = \STARSTAR{Y} \times \STAR{[\LP{1}([0,h];\STAR{Y})]}  
\end{displaymath}
see \S \ref{sec:setting:sunstar}, and, as remarked there, $\STAR{[\LP{1}([0,h];\STAR{Y})]} \neq \LP{\infty}([-h,0];\STARSTAR{Y})$. Moreover, a representation of the second dual space $\STARSTAR{Y}$ is generally unknown, e.g. when $Y = \CONT(\BOmega)$ as for \eqref{eq:neuralfield:10}. 
\begin{remark}
  \label{rem:normalforms:1}
  In \S \ref{sec:numerics} it will turn out that the second derivative $B$ in \eqref{eq:normalforms:3} vanishes due to a symmetry in \eqref{eq:neuralfield:10} for the particular modelling functions chosen. In the present subsection we deliberately do not exploit this information in order to illustrate a general principle. \hfill \QEDEX
\end{remark}
In this subsection we offer a way around these complications that works for equations of the type \eqref{eq:neuralfield:8}. We first deal with the problem of determining $R(\lambda,\SUNSTAR{A})$. From Lemma \ref{lem:setting:5} in \S \ref{sec:setting:sunstar} it follows that the second and third derivatives defined in \eqref{eq:normalforms:3} and \eqref{eq:normalforms:4}, as well as all derivatives of higher order, map into the closed subspace $Y \times \{0\}$ of $\SUNSTAR{X}$. By inspection of the expressions for the coefficients $h_{\nu}$ in \S\S \ref{sec:normalforms:hopf} and \ref{sec:normalforms:double-hopf} one sees that it is sufficient to obtain a representation of the action of $R(\lambda,\SUNSTAR{A})$ on this space.
\begin{lemma}
  \label{lem:normalforms:2}
  Suppose that $\lambda \in \rho(A)$. For each $y \in Y$ the function $\psi = \EPS_{\lambda} \otimes \Delta(\lambda)^{-1}y$ is the unique solution in $\CONTD{1}([-h,0];Y)$ of the system
  \begin{equation}
    \label{eq:normalforms:2}
    \left\{
      \begin{aligned}
        \lambda \psi(0) - DF(0)\psi &= y\\
        \lambda \psi - \psi'&= 0
      \end{aligned}
    \right.
  \end{equation}
  Moreover, $\SUNSTAR{\psi} = j\psi$ is the unique solution in $D(\SUNSTAR{A})$ of $(\lambda - \SUNSTAR{A})\SUNSTAR{\psi} = (y,0)$.
\end{lemma}
\begin{proof}
  We return to the setting of Proposition \ref{prop:spectrum:1} in \S \ref{sec:spectrum:generalities} with $L = DG(0)$. Since $\lambda \in \rho(A)$ it follows that $\Delta(\lambda)^{-1}$ exists. We start by showing that $\psi = \EPS_{\lambda} \otimes \Delta(\lambda)^{-1}y$ solves \eqref{eq:normalforms:2}. Explicitly,
  \begin{displaymath}
    \psi(\theta) = e^{\lambda\theta} \Delta(\lambda)^{-1}y \qquad \forall\,\theta \in [-h,0]
  \end{displaymath}
  so clearly $\psi \in \CONTD{1}([-h,0];Y)$ and $\psi$ satisfies the second equation in \eqref{eq:normalforms:2}. From \eqref{eq:setting:6} we recall that $DF(0)\psi = -\alpha\psi(0) + DG(0)\psi$. Therefore,
  \begin{align*}
    \lambda\psi(0) - DF(0)\psi &= (\lambda + \alpha)\psi(0) - DG(0)\psi\\
    &= (\lambda + \alpha)\Delta(\lambda)^{-1}y - DG(0)(\EPS_{\lambda} \otimes \Delta(\lambda)^{-1}y)\\
    &= (\lambda + \alpha)\Delta(\lambda)^{-1}y - L_{\lambda}(\Delta(\lambda)^{-1}y)\\
    &= \Delta(\lambda)\Delta(\lambda)^{-1}y = y
  \end{align*}
  Lemma \ref{lem:setting:4} in \S \ref{sec:setting:sunstar} implies that $j\psi \in D(\SUNSTAR{A})$, where $j$ is the embedding defined in \eqref{eq:setting:18}, and
  \begin{displaymath}
    (\lambda - \SUNSTAR{A})j\psi = \lambda
    \left[
      \begin{array}{l}
        \psi(0)\\
        \psi
      \end{array}
    \right] - 
    \begin{bmatrix}
      DF(0)\psi\\
      \psi'
    \end{bmatrix} = (y,0)
  \end{displaymath}
  But $\sigma(\SUNSTAR{A}) = \sigma(A)$ so $\SUNSTAR{\psi} = j\psi$ is the \emph{unique} solution of $(\lambda - \SUNSTAR{A})\SUNSTAR{\psi} = (y,0)$. Consequently, $\psi$ itself is the unique solution in $\CONTD{1}([-h,0];Y)$ of \eqref{eq:normalforms:2}.
\end{proof}
The above lemma takes care of one of the two problems sketched above. Now suppose that $\lambda \in \sigma(A) \setminus \{-\alpha\}$ is a simple eigenvalue with eigenvector $\phi \in D(A)$. (Note that $\lambda$ is isolated in $\sigma(A)$ by Corollary \ref{cor:spectrum:3} in \S \ref{sec:spectrum:generalities}.) Let $\SUN{\phi} \in D(\STAR{A})$ be a corresponding eigenvector of $\STAR{A}$. Without loss of generality we may assume that $\PAIR{\phi}{\SUN{\phi}} = 1$. Let $\SUN{P}$ and $\SUNSTAR{P}$ be the associated spectral projections on $\SUN{X}$ and $\SUNSTAR{X}$, respectively. We set out to evaluate $\PAIR{\SUN{\phi}}{\SUNSTAR{\phi}}$ where $\SUNSTAR{\phi} = (y,0) \in Y \times \{0\} \subseteq \SUNSTAR{X}$ is given, but $\SUN{\phi}$ is unknown. Since the range of $\SUNSTAR{P}$ is spanned by $j\phi$ we have $\SUNSTAR{P}\SUNSTAR{\phi} = \kappa j\phi$ for a certain $\kappa \in \CC$. In fact, by \eqref{eq:setting:18} it follows that
\begin{equation}
  \label{eq:normalforms:6}
  \PAIR{\SUN{\phi}}{\SUNSTAR{\phi}} = \PAIR{\SUN{P}\SUN{\phi}}{\SUNSTAR{\phi}} = \PAIR{\SUN{\phi}}{\SUNSTAR{P}\SUNSTAR{\phi}} = \kappa \PAIR{\SUN{\phi}}{j\phi} = \kappa
\end{equation}
so $\kappa$ is to be determined. This may be done as follows. From the Cauchy integral representation \cite[\S 5.8]{Taylor1958} for $\SUNSTAR{P}_{\lambda}$ we infer that
\begin{equation}
  \label{eq:normalforms:5}
  \SUNSTAR{P}\SUNSTAR{\phi} = \frac{1}{2\pi i}\oint_{\DD C_{\lambda}}{R(z,\SUNSTAR{A})\SUNSTAR{\phi}\,dz} = \kappa j\phi
\end{equation}
where $C_{\lambda}$ is any open disk centered at $\lambda$ such that $\BAR{C}_{\lambda,0} \subseteq \rho(A)$ where $C_{\lambda,0} \DEF C_{\lambda} \setminus \{\lambda\}$ and $\DD C_{\lambda}$ is its boundary. Since $\SUNSTAR{\phi} \in Y \times \{0\}$ the integrand in \eqref{eq:normalforms:5} may be calculated using Lemma \ref{lem:normalforms:2}. Specifically, for $z \in \DD C_{\lambda}$ we have
\begin{displaymath}
  R(z,\SUNSTAR{A})\SUNSTAR{\phi} = j(\EPS_z \otimes \Delta(z)^{-1}y) =
  \left[
    \begin{array}{r}
      \Delta(z)^{-1}y\\
      \EPS_z \otimes \Delta(z)^{-1}y
    \end{array}
  \right]
\end{displaymath}
Since $j\phi = \phi(0)$ we may restrict our attention to the first component to infer that
\begin{equation}
  \label{eq:normalforms:7}
  \frac{1}{2\pi i}\oint_{\DD C_{\lambda}}{\Delta(z)^{-1}y\,dz} = \kappa \phi(0)
\end{equation}
We note that the integral is $Y$-valued and \eqref{eq:normalforms:7} is an identity in $Y$. The integrand may be evaluated using the results of \S \ref{sec:spectrum:resolvent}. Indeed, for each $z \in \DD C_{\lambda}$ it is necessary to solve a system of the type \eqref{eq:spectrum:22}, but with $S_{z}\phi$ replaced by $y$.
\par
For this purpose we may apply Theorem \ref{thm:spectrum:1} as follows. Let us assume that $\lambda$ is a root of the characteristic equation \eqref{CharacteristicEquation} on the imaginary axis, $\lambda \not\in \mathcal{S}$, the roots $\pm \rho_i(\lambda)$ of the polynomial \eqref{ODE01CharPolynom2} are all distinct and \eqref{eq:spectrum:15} holds. Suppose it has also been verified that the matrix $T(\lambda)$ is invertible. By choosing the radius of $C_{\lambda}$ sufficiently small, we guarantee that for every $z \in C_{\lambda}$ it holds that $z \neq -\alpha$, $z \not\in \mathcal{S}$, the roots $\pm \rho_i(z)$ are all distinct, \eqref{eq:spectrum:15} is satisfied (with $z$ instead of $\lambda$) and $T(z)$ is invertible. Furthermore, in this way we may also ensure that $z \in \rho(A)$ for every $z \in C_{\lambda,0}$. In particular, all the maps
\begin{displaymath}
  C_{\lambda} \ni z \mapsto \pm\rho_i(z) \in \CC \qquad (i = 1,\ldots,N)
\end{displaymath}
are analytic. By \eqref{eq:spectrum:18} and \eqref{GammaDefinition01} (with $y$ in place of $h_{z}$) this implies that
\begin{displaymath}
  C_{\lambda} \ni z \mapsto \hat{\Gamma}_z(x) \in \CC^{2N} 
\end{displaymath}
is analytic for all $x \in \BOmega$. Hence by \eqref{eq:spectrum:23} (with $y$ instead of $h_{z}$) we have, for every $x \in \BOmega$,
\begin{align}
	\oint_{\DD C_{\lambda}}{[\Delta(z)^{-1}y](x)\,dz} &= \oint_{\DD C_{\lambda}}{\frac{y(x)}{z+\alpha}+\sum_{i=1}^{N}{\bigl[\gamma_{i,z}(x)e^{\rho_i(z)x} + \gamma_{-i,z}(x)e^{-\rho_i(z) x}\bigr]} \,dz}\nonumber\\
	&= \sum_{i=1}^{N}{\left[ e^{\rho_i(\lambda)x}\oint_{\DD C_{\lambda}}{[\Gamma_{0,z}]_i \,dz} + e^{\rho_{-i}(\lambda)x} \oint_{\DD C_\lambda}{[\Gamma_{0,z}]_{-i} \,dz}\right]}\label{eq:normalforms:8}
\end{align} 
where $\Gamma_{0,z}$ is as in \eqref{Gamma0Definition}. Since $\Gamma_{0,z}$ involves $S(z)^{-1}$, the maps
\begin{displaymath}
  C_{\lambda} \ni z \mapsto [\Gamma_{0,z}]_{\pm i} \in \CC \qquad (i = 1,\ldots,N)
\end{displaymath}
cannot be expected to be analytic and \eqref{eq:normalforms:8} may not be reduced further. 
\par
In summary, \eqref{eq:normalforms:8} provides a way to evaluate \eqref{eq:normalforms:7} by numerical integration. It suffices to parametrise $\DD C_\lambda$ and apply a quadrature rule to compute the $\CC$-valued contour integrals
\begin{displaymath}
  \oint_{\DD C_{\lambda}}{[\Gamma_{0,z}]_{\pm i} \,dz} \qquad (i = 1,\ldots,N)
\end{displaymath}
which are independent of $x \in \BOmega$. We then verify that $\tfrac{1}{2\pi i}$ times \eqref{eq:normalforms:8} and $\phi(0)$ are indeed proportional to each other as functions of $x \in \BOmega$. The value of $\PAIR{\SUN{\phi}}{\SUNSTAR{\phi}}$ in \eqref{eq:normalforms:6} then equals the corresponding proportionality constant $\kappa$.


\section{Numerical calculations}\label{sec:numerics}
In \S \ref{sec:spectrum:chareq} we derived a characteristic equation for problem \eqref{eq:spectrum:4} under the assumption that $J$ is a finite linear combination of exponentials. The main result was formulated in Theorem \ref{thm:CharacteristicEquation}. Subsequently, in \S \ref{sec:spectrum:resolvent} we obtained a closed expression for the associated resolvent operator. In the present section we apply these findings together with the theory from \S \ref{sec:cm} to a concrete example. For reasons that will become apparent later, we assume that the connectivity function has a bi-exponential form,
\begin{equation}
  \label{eq:numerics:1}
  J(x,r) = \hat{c}_1 e^{-\mu_1 |x-r|} + \hat{c}_2 e^{-\mu_2 |x-r|} \qquad \forall\,x,r \in \BOmega
\end{equation}
and we choose the activation function $S$ as in \cite{Faye2010},
\begin{displaymath}
  S(V) = \frac{1}{1 + e^{-r V}} - \frac{1}{2} \qquad \forall\,V \in \RR
\end{displaymath}
Since $S(0)=0$ it follows that \eqref{eq:neuralfield:10} admits the trivial steady state $V \equiv 0$ on which we will focus from now on. Here we have $S'(0) = \tfrac{r}{4}$ and hence $c_i = \frac{r}{4}\hat{c}_i$ for $i=1,2$.
\par
Let us continue by expressing the characteristic equation for this example and discussing a naive approach for finding its roots. Thereafter we compare these results with a more traditional approach which discretises the spatial domain $\BOmega$. Such a discretisation can be studied using techniques and software that are already available. We conclude with a normal form analysis of a Hopf bifurcation and a double Hopf bifurcation to illustrate the potential of the results from \S \ref{sec:cm}.

\subsection{Spectral calculations}\label{sec:numerics:spectrum}
In order to apply Theorem \ref{thm:CharacteristicEquation}, we start by considering the characteristic polynomial $\mathcal{P}$ from \eqref{ODE01CharPolynom2}, which presently takes the form
\small
\begin{align*}
  \mathcal{P}(\rho) = \frac{e^{\lambda\tau_0}(\lambda+\alpha)}{2} (\rho^2 &-(\lambda+\mu_1)^2)(\rho^2-(\lambda+\mu_2)^2)\\
  &+ c_1(\lambda+\mu_1)(\rho^2-(\lambda+\mu_2)^2) + c_2(\lambda+\mu_2)(\rho^2-(\lambda+\mu_1)^2)
\end{align*}
\normalsize
This is a second order polynomial in $\rho^2$. We apply a Newton algorithm to the mapping $\lambda \mapsto \DET{S(\lambda)}$ to find the solutions of the characteristic equation \eqref{CharacteristicEquation}. At each root $\hat{\lambda}$ we need to verify that $\hat{\lambda} \not\in \mathcal{S}$, the numbers $\pm \rho_{1,2}(\hat{\lambda})$ are all distinct and \eqref{eq:spectrum:15} is satisfied. (Note that both of these are \emph{open} conditions.) Passing this test we may conclude that $\hat{\lambda}$ is indeed an eigenvalue.

\subsection{Discretisation}
\subsubsection*{Derivation \cite{Faye2010}}
An approximate solution to the neural field equation can be obtained by discretising \eqref{eq:neuralfield:10}. This reduces the state space from $\CONT([-h,0];Y)$ to $\CONT([-h,0];\RR^{m+1})$ for some $m \in \NN$. Hence the theory of `classical' DDE can be applied to analyse the approximate system. 
\par
We heuristically derive the discretised system following \cite{Faye2010} but we make a few minor corrections. Consider the original equation \eqref{eq:neuralfield:7}:
\begin{displaymath}
  \frac{\DD V}{\DD t}(t,x) = -\alpha V(t,x) + \sum_{i=1}^{m}{\int_{x_{i-1}}^{x_i}{J(x,r)S(V(t-\tau(x,r),r))\,dr}}
\end{displaymath}
for some partition $-1=x_0 < x_1 < \ldots < x_m=1$. We approximate every single integral with a two-point trapezoid rule evaluated at the end points of the integration interval,
\begin{align*}
  \frac{\partial V}{\partial t} (t,x) \approx -\alpha V(t,x) + \sum_{i=1}^{m}\frac{x_{i}-x_{i-1}}{2}\bigl[&J(x,x_{i-1})S(V(t-\tau(x,x_{i-1}),x_{i-1}))\\
  + &J(x,x_i)S(V(t-\tau(x,x_i),x_i))\bigr]
\end{align*}
By writing $V_j(t) = V(t,x_j)$, we obtain for $j=0,1,\ldots,m$,
\begin{align*}
  \frac{dV_j}{dt}(t) = -\alpha V_j(t) + \sum_{i=1}^{m} \frac{x_{i}-x_{i-1}}{2}\bigl[&J(x_j,x_{i-1})S(V_{i-1}(t-\tau(x_j,x_{i-1})))\\
  + &J(x_j,x_i)S(V_{i}(t-\tau(x_j,x_i)))\bigr]
\end{align*}
As in \eqref{eq:spectrum:19} we take $\tau(x,r) = \tau_0 + |x-r|$. Also, with some abuse of notation we write $J(|x-r|)$ for $J(x,r)$, since the dependence in the right-hand side of \eqref{eq:numerics:1} on $x,r$ is only via $|x - r|$. By restriction to an equidistant mesh of size $\delta = x_{i}-x_{i-1} = \tfrac{2}{m}$, we obtain
\begin{align*}
  \frac{dV_j}{dt}(t) = -\alpha V_j(t) + \frac{2}{m}\sum_{i=1}^{m}\frac{1}{2}\bigl[&J(\delta|i-j-1|)S(V_{i-1}(t-\tau_0-\delta|i-j-1|))\\
  + &J(\delta|i-j|)S(V_{i}(t-\tau_0-\delta|i-j|))\bigr]
\end{align*}
Defining
\begin{displaymath}
  w_i=\begin{cases}
    \frac{1}{2} &\text{if } i\in\{0,m\} \\
    1	&\text{if } i\in\{1,2,\ldots,m-1\}
  \end{cases}
\end{displaymath}
enables us to telescope the summation, arriving at
\begin{equation}\tag{DNF}
  \label{DiscretizedDefinition}
  \frac{dV_j}{dt}(t) = -\alpha V_j(t) + \frac{2}{m}\sum_{i=0}^{m}{w_i J(\delta|i-j|)S(V_{i}(t-\tau_0-\delta|i-j|))}
\end{equation}
for all $j=0,1,\ldots,m$. We refer to \eqref{DiscretizedDefinition} as the \emph{discretisation} of \eqref{eq:neuralfield:10} or \eqref{eq:neuralfield:7}. Note that \eqref{DiscretizedDefinition} indeed is a classical DDE, albeit with many delays, which may be implemented in \verb#MATLAB# to perform forward-time simulations using the {\tt dde23} scheme. In particular, the software package \verb#DDE-BIFTOOL# \cite{Engelborghs2002} allows us to determine the spectrum of the discretised system. At the end of this section we consider two examples in which we use both our analytic results and these numerical tools to study critical points in neural fields.

\subsubsection*{Convergence of discretisation}
In order to `validate' the above discretisation procedure, we generate discretisations with different resolutions and compare their spectra with the spectral values obtained by using the methods from \S \ref{sec:spectrum:chareq}. This is illustrated in Figure \ref{fig:SpectrumConvergence}.
\begin{figure}
  \centering
  \includegraphics[width=0.85\textwidth]{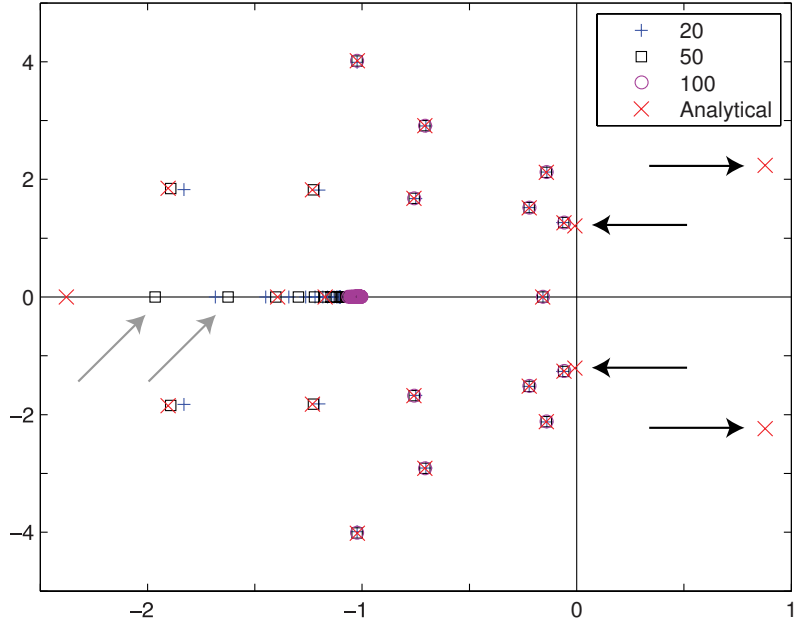}
  \caption{Comparison between spectra of the discretised system for $m=20,50,100$ and roots of the characteristic equation. The four black arrows indicate four roots which are not in the spectrum and the grey arrows point out distinct values which are not found. See text for a more elaborate description of these points. $\alpha = 1, \tau_0=1, c_1 = -5, c_2 = 2, \mu_1 = 2, \mu_2 = 0$.}
  \label{fig:SpectrumConvergence}
\end{figure}
\begin{figure}
  \centering
  \includegraphics[width=0.75\textwidth]{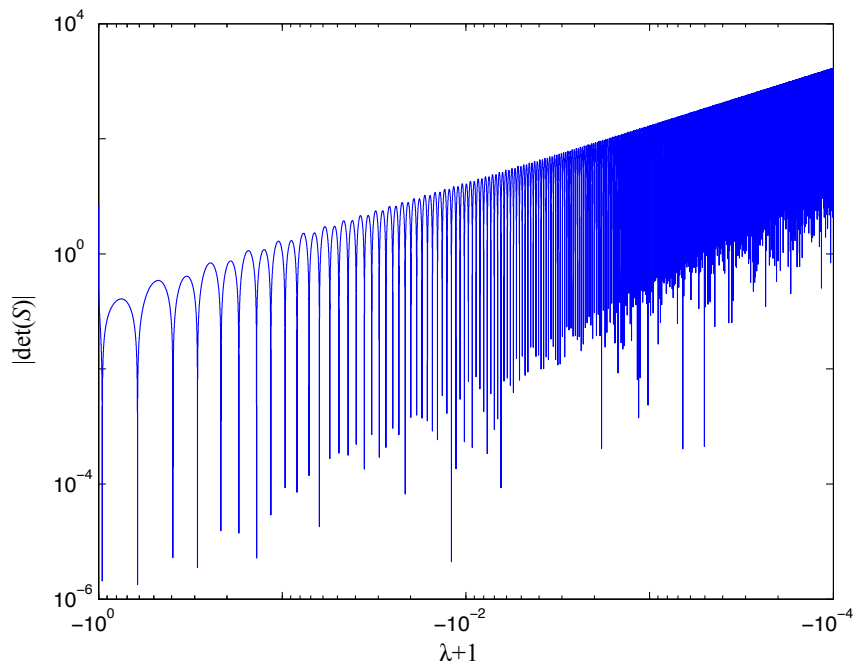}
  \caption{Detail of the accumulation of eigenvalues along the real line near $\lambda=-1$: $|\det(S)|$ is plotted near the accumulation point and downward peaks correspond to roots of the characteristic equation.}
  \label{fig:SpectrumAccumulation}
\end{figure}
The black arrows indicate four roots of the characteristic equation which do not satisfy all conditions stated in Theorem \ref{thm:CharacteristicEquation}: these points do satisfy \eqref{CharacteristicEquation}, but also $\rho_1(\hat{\lambda}) = \rho_2(\hat{\lambda})$. 
Therefore they are to be rejected as eigenvalues. 
Furthermore, the grey arrows indicate eigenvalues which are not found using the algorithm of \S \ref{sec:numerics:spectrum}. 
These eigenvalues all lie in the `accumulation region' near the point $-\alpha$. 
\par
We proceed by studying this accumulation of eigenvalues more closely. For $\lambda\uparrow-1$ along the real axis the absolute value of $\DET{S(\lambda)}$ is plotted in Figure \ref{fig:SpectrumAccumulation} on log-log scale. Every downward peak corresponds to a root of the characteristic equation. For $\lambda$ close to $-\alpha$ the numerical accuracy drops causing the peaks to be less pronounced. This shows that, while these are not located by the root finder algorithm, the characteristic equation does have accumulating roots near the essential spectrum $\{-\alpha\}$ as is suggested by the spectrum of the discretisation, cf. Figure \ref{fig:SpectrumConvergence}.
However, due to both the high frequency oscillations and numerical errors, the Newton algorithm is unable to locate these roots.
\par
Finally we observe that spectra corresponding to finer meshes converge to the analytic spectrum. However, it appears that for increasing resolutions \verb#DDE-BIFTOOL# focuses on roots near $-\alpha$ instead of eigenvalues located further away. This is clearly seen when $m=100$, in which case no spectral values $\lambda$ are found with $\RE{\lambda} > -1.2$. For that reason we have chosen $m=50$ in the following examples.
\subsection{Hopf bifurcation}\label{sec:numerics:hopf}
Rhythms and oscillations are important features of nervous tissue that could be studied with neural field models. For that reason Hopf bifurcations play a key role in the analysis of neural field equations. In this subsection we study a concrete example of a Hopf bifurcation, both analytically and numerically. We also compare the results of both methods.

Initially we focus on a connectivity of the `inverted wizard hat'-type. Similarly as in \cite{Faye2010} we consider the steepness parameter $r$ of the activation function as bifurcation parameter.

\subsubsection*{Spectrum}
The characteristic equation \eqref{CharacteristicEquation} is used to determine the point spectrum for a range of parameters. For the values shown in Table \ref{tbl:H_Param} there exists a purely imaginary pair of simple eigenvalues. The corresponding spectrum is displayed in Figure \ref{fig:H_Spectrum}.
\def\arraystretch{1.5}
\begin{table}
  \centering
  \caption{Parameters corresponding to Hopf bifurcation}
  \label{tbl:H_Param}
  \begin{tabular}{l|ccccccc}
    parameter & $\alpha$ & $\tau_0$ & $r$ & $\hat{c}_1$ & $\hat{c}_2$ & $\mu_1$ & $\mu_2$\\
    \hline
    value     & $1.0$    & $1.0$  & $4.220214885988226$ & $3.0$ & $-5.5$ & $0.5$ & $1.0$ 
  \end{tabular}
\end{table}
\def\arraystretch{1.0}
\begin{figure}
  \centering
  \includegraphics[width=0.70\textwidth]{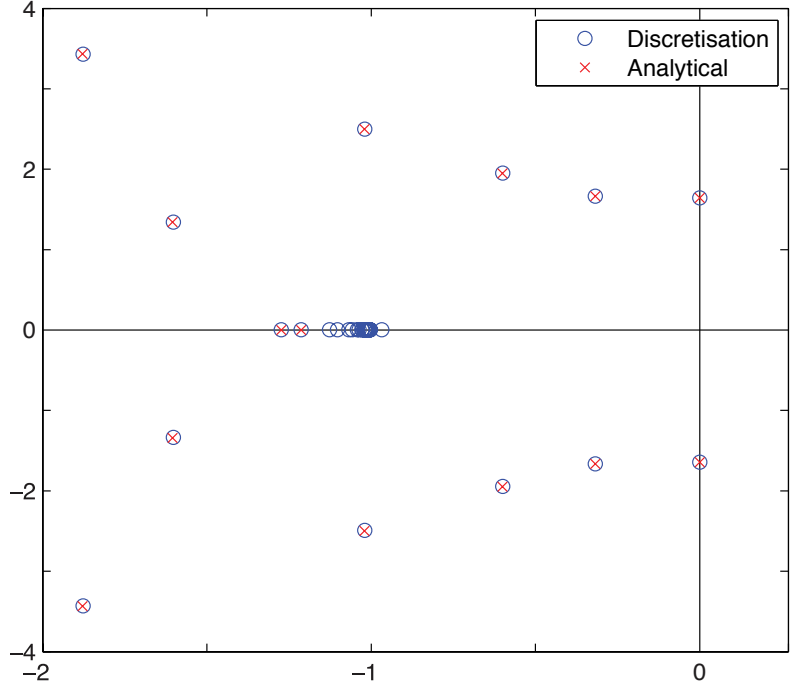}
  \caption{Spectrum at a Hopf bifurcation. Comparison between analytic approach and discretised system ($m=50$).}
  \label{fig:H_Spectrum}
\end{figure}
The figure also shows the spectrum as calculated by \verb#DDE-BIFTOOL# for a discretisation of $m=50$ intervals. From the graph it is obvious that the solution algorithm of \S \ref{sec:numerics:spectrum} is unable to locate eigenvalues near the accumulation point $-\alpha=-1$. We discuss this phenomenon below. Apart from that, the numerical approximation seems to be very close to the analytic solution. Only in the far left half-plane an error can be observed. From a dynamical point of view such an error is of course rather innocuous.

\subsubsection*{First Lyapunov coefficient}
In order to determine analytically the type of Hopf bifurcation (i.e. sub- or supercritical), the first Lyapunov coefficient has to be determined. Before the result of \S \ref{sec:normalforms:hopf} can be applied, the eigenfunction corresponding to the eigenvalues at criticality has to be determined. Application of Theorem \ref{thm:CharacteristicEquation} yields
\begin{displaymath}
  \phi(t,x) = e^{\lambda t} \bigl[\gamma_1(e^{\rho_1 x}+e^{-\rho_1 x}) + \gamma_2(e^{\rho_2 x}+e^{-\rho_2 x}) \bigr] \qquad \forall\,t \in [-h,0],\,\forall\,x \in \BOmega
\end{displaymath}
with
\begin{align*}
  \rho_1 &= 0.321607348361597 - 0.880461478656249i\\
  \rho_2 &= 0.110838003673357 - 2.312123026384049i\\
  \gamma_1 &=-0.191821747840362 - 0.172140605861736i\\
  \gamma_2 &=-0.080160108888561
\end{align*}
corresponding to $\lambda=1.644003102046893i$. (Note that in the present example with $\tau_0$ as in Table \ref{tbl:HH_Param} and $\BOmega = [-1,1]$ the delay interval equals $[-h,0] = [-3,0]$.) Since the activation function $S$ is odd, its second derivative vanishes and the critical normal form coefficient $g_{21}$ in \eqref{Hopf.c1} significantly simplifies to
\begin{displaymath}
  g_{21}=\frac{1}{2} \PAIR{\SUN{\phi}}{D^3R(0)(\phi,\phi,\BAR{\phi})}
\end{displaymath}
The pairing is expressed as a contour integral around $\lambda$ which we evaluate numerically, see \S \ref{sec:normalforms:calc}. We find $g_{21} \approx -0.326+0.0389i$ and hence the first Lyapunov coefficient is $l_1 \approx -0.198$. The negative sign of $l_1$ indicates a \emph{supercritical} Hopf bifurcation. Therefore stable periodic solutions are expected to emerge from the bifurcating steady state.

\subsubsection*{Simulations}
We choose $r=6$ which is beyond the critical parameter value of Table \ref{tbl:H_Param}. For the initial condition $V(t,x)=\EPS=0.01$ for $t\in[-h,0]$ with $h=3$ and $x\in[-1,1]$ the simulation result is shown in Figure \ref{fig:H_Simulations}. After a transient time, the system approaches its stable periodic attractor. The convergence to stable periodic motion is consistent with the sign of the first Lyapunov coefficient. Furthermore, both the shape and period of this attractor match with the eigenfunction and eigenvalue respectively.
\begin{figure}
  \centering
  \includegraphics[width=\textwidth]{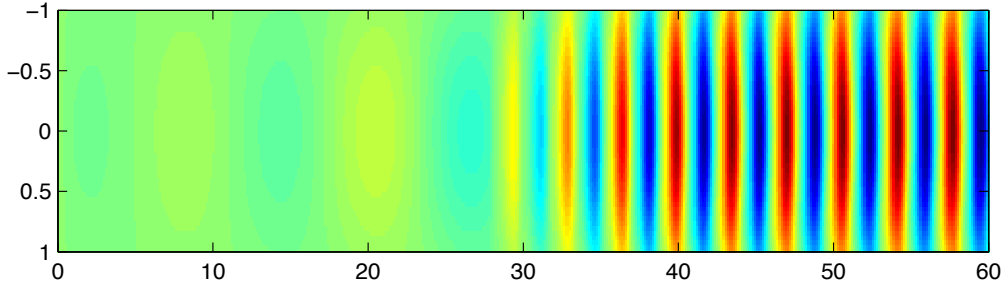}
  \caption{Forward time simulation of discretised system ($m=50$) for $r=6$ beyond a Hopf bifurcation. A long transient is observed before the solution approaches the limit cycle.}
  \label{fig:H_Simulations}
\end{figure}

\subsection{Double Hopf bifurcation}\label{sec:numerics:dh}
The spectrum at the Hopf point studied in \S \ref{sec:numerics:hopf} consists mainly of complex pairs of eigenvalues. Therefore it is to be expected that system parameters can be tuned such that a second pair of complex eigenvalues arrives at the imaginary axis, giving rise to a double Hopf bifurcation. In this subsection we show that this is indeed possible and we study this bifurcation both analytically and numerically.

\subsubsection*{Spectrum}
\def\arraystretch{1.5}
\begin{table}
  \centering
  \caption{Parameters corresponding to double Hopf bifurcation.}
  \label{tbl:HH_Param}  
  \begin{tabular}{l|ccccccc}
    parameter & $\alpha$ & $\tau_0$ & $r$ & $\hat{c}_1$ & $\hat{c}_2$ & $\mu_1$ & $\mu_2$\\
    \hline
    value     & $1.0$    & $1.0$  & $4.828749714457348$ & $3.0$ & $-5.5$ & $0.0$ & $0.999592391420082$ 
  \end{tabular}
\end{table}
\def\arraystretch{1.0}
\def\arraystretch{1.5}
\begin{table}
  \centering
  \caption{Values of $\lambda$ and corresponding $\rho$}
  \label{tbl:HH_Lambda}
  \begin{tabular}{c|r@{$-$}l}
    $\lambda$ & \multicolumn{2}{c}{$\rho(\lambda)$} \\
    \hline
    $2.030930500644927i$ & $0.454550410967142$ & $1.057267648955222i$ \\
    & $0.054136932895367$ & $3.495632804443535i$\\
    \hline
    $1.299147304907829i$ & $1.075429529957343$ & $0.717519976488838i$ \\
    & $1.128716151852882$ & $2.306528729845143i$ \\
  \end{tabular}
\end{table}
\def\arraystretch{1.0}
\begin{figure}
  \centering
  \includegraphics[width=0.70\textwidth]{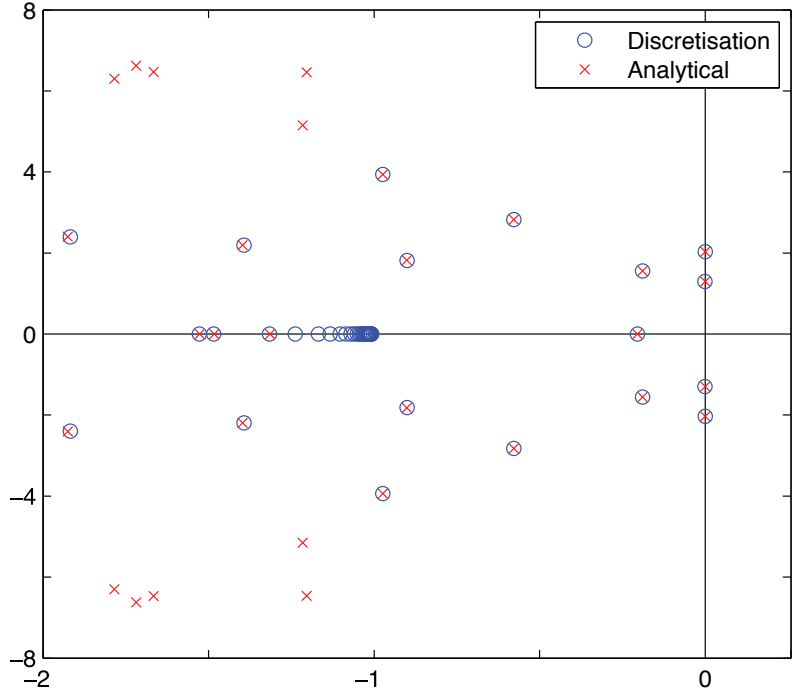}
  \caption{Spectrum at a double Hopf bifurcation. Comparison between analytic approach and discretised system ($m=50$).}
  \label{fig:HH_Spectrum}
\end{figure}
Parameters for which the system has two complex pairs of eigenvalues on the imaginary axis and no eigenvalues in the positive right half-plane are identified, see Table \ref{tbl:HH_Param}. The corresponding spectrum is depicted in Figure \ref{fig:HH_Spectrum} while Table \ref{tbl:HH_Lambda} lists the values of $\lambda$ at this critical point. As with the regular Hopf bifurcation, we observe that the root finding algorithm misses most eigenvalues near the essential spectrum at $-\alpha=-1$.
\par
Next we compute the eigenfunctions corresponding to the critical eigenvalues. For this purpose Theorem \ref{thm:CharacteristicEquation} may be applied using the data in Table \ref{tbl:HH_Lambda}. Modulus and argument of both eigenfunctions are depicted in Figure \ref{fig:HH_EigenFuncs}.
\begin{figure}
  \centering
  \includegraphics[width=0.80\textwidth]{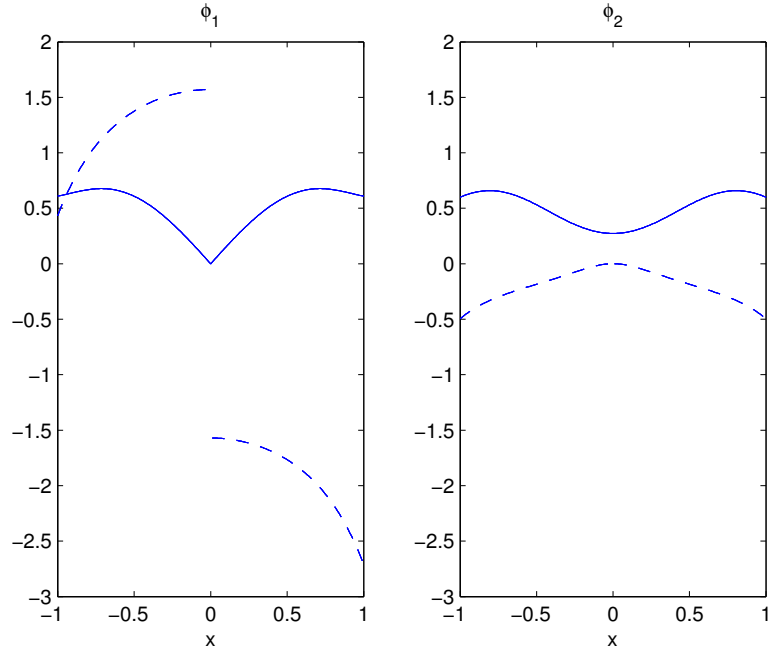}
  \caption{Eigenfunctions corresponding to the critical eigenvalues. $\lambda=i1.299...$ is shown left and $\lambda=i2.030...$ on the right. Solid lines depict the modulus and dashed lines the argument.}
  \label{fig:HH_EigenFuncs}
\end{figure}

\subsubsection*{Normal form coefficients}
With this information available, the normal form coefficients of the double Hopf bifurcation may be evaluated. The coefficients $g_{2100}$, $g_{1011}$, $g_{1110}$, and $g_{0021}$ from \S \ref{sec:normalforms:double-hopf} are found as in \S \ref{sec:numerics:hopf}, which results in the matrix
\begin{displaymath}
  \begin{bmatrix}
    p_{11} & p_{12} \\
    p_{21} & p_{22}
  \end{bmatrix}
  =
  \begin{bmatrix}
    -8.822 & -3.367 \\
    -13.79 & -1.310
  \end{bmatrix}
\end{displaymath}
Since $p_{11}p_{22}>0$, we conclude that this double Hopf bifurcation is of the `simple' type, see \cite[\S 8.6.2]{Kuznetsov2004}. Defining $\theta \DEF \tfrac{p_{12}}{p_{22}} \approx 2.57$ and $\delta \DEF \tfrac{p_{21}}{p_{11}} \approx 1.56$, we find that $\theta\delta>1$ and therefore this `simple' bifurcation has sub-type I \cite{Kuznetsov2004}. The key feature of this sub-type is the presence of a regime in parameter space for which two distinct stable periodic solutions exist.

\subsubsection*{Simulations}
\begin{figure}
  \centering
  \includegraphics[width=0.80\textwidth]{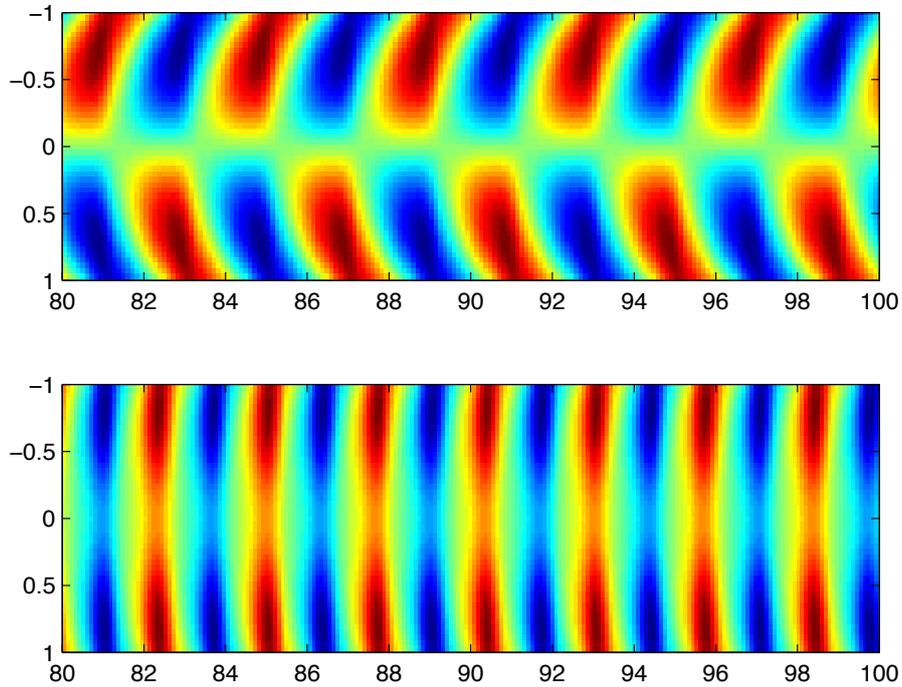}
  \caption{Bi-stability near double Hopf bifurcation: for $r=6$ and $\mu_2=1$ the time evolution is shown for different initial conditions ($m=50$). Top and bottom diagrams correspond to \eqref{HH_SimInits1} and \eqref{HH_SimInits2} respectively.}
  \label{fig:HH_Simulations}
\end{figure}
The parameters are adjusted slightly, such that both pairs of complex eigenvalues have a positive real part. More specifically we choose $r=6$ and $\mu_2=1$ while keeping other parameters as in Table \ref{tbl:HH_Param}. For the following two initial conditions
\begin{subequations} 
  \begin{align}
    V(t,x) &= \EPS x \label{HH_SimInits1} \\
    V(t,x) &= \EPS \label{HH_SimInits2}
  \end{align}
\end{subequations}
with $\EPS=0.01$, the discretised system ($m=50$) is integrated forwards in time, see Figure \ref{fig:HH_Simulations}.
\par
For the chosen parameters the system has two stable periodic attractors and the asymptotic behaviour is determined by their initial conditions. This result is consistent with the predictions of the normal form computation. Furthermore we observe, since the system is close to the double Hopf bifurcation, that both the shape and period of the periodic solutions are fairly well approximated by the critical eigenfunctions (c.f. Figure \ref{fig:HH_EigenFuncs}). Indeed, the moduli of the eigenfunctions correspond to the amplitude of the asymptotic solution. For either of the solutions the extrema are located near the borders of the domain. The antiphasic solutions in the upper panel are indicated by the jump of size $\pi$ in the argument of the first eigenfunction, see Figure \ref{fig:HH_EigenFuncs} on the left.


\section{Discussion}\label{sec:neuralfield:discussion}
We have demonstrated that neural field equations with transmission delay fit well into the sun-star framework for delay equations. As a consequence, standard results from dynamical systems theory, such as the principle of linearised (in)stability, center manifold reduction and normal form computation, are readily available. These, in turn, open the possibility for a systematic study of codimension one and two local bifurcations, w.r.t.  parameters in the connectivity and activation functions. This facilitates an understanding of the effect of parameters in terms of biological quantities. 
\par
In \S \ref{sec:numerics} we analysed the dynamics of a one population model with the inverted {\em wizard} hat as connectivity function. The choice of an inverted {\em Mexican} hat is biologically more plausible, as pyramidal cells are surrounded by a cloud of interneurons, while long range connections are by and large excitatory. We have chosen the inverted wizard hat mainly for mathematical convenience. Indeed, in \S \ref{sec:spectrum:chareq} an analytic formula for the location of the eigenvalues was derived. It is well known that the combination of an inverted Mexican hat connectivity with a transmission delay leads to dynamic instabilities \cite{Bressloff1996,Hutt2003}. In \cite{Coombes2005} Turing instabilities were shown to occur for the inverted wizard hat connectivity.
\par
We have seen that the stationary spatially homogeneous state destabilises upon increasing the steepness (gain) of the activation function. This is in line with other findings indicating that the activation function strongly influences dynamical behaviour, see for instance \cite{Ermentrout2010,Coombes2010}. 
\par
It is mathematically challenging to consider neural field equations on unbounded spatial domains, leading to infinite delays, although such is of less importance from a biological viewpoint. Our main goal for the near future is to develop tools for numerical bifurcation analysis for the class of equations studied in this paper. Normal form computation is a first prerequisite for this task. Hence we are on our way.


\appendix
\section{Proof of Proposition \ref{prp:ODE01CharEq}}\label{app:ProofCharPolynom}
We use the same notation as in Lemma \ref{lmm:ODEEquivalence} and its proof. We recall that the vector $Z = [\zeta_0, \zeta_1,\ldots,\zeta_{N-1},1]$ is chosen such that the vector $\beta = [\beta_0,\beta_1,\ldots,\beta_N]$, whose elements are the coefficients of the characteristic polynomial $\mathcal{P}$, is given by $\beta = M^TZ$. Introducing $r \DEF [1,\rho^2, \rho^4,\ldots,\rho^{2N}]$ we see that
\begin{equation}\label{Pdef}
  \mathcal{P}(\rho) = r^T M^T Z
\end{equation}
First we determine the vector $Z$, thereafter we split $M$, and we conclude the proof by determining how $Z$ acts on each factor in this decomposition.
\par
Although $Z$ can be obtained by applying the inverse of the Vandermonde matrix $W$, we will proceed in a different manner. We start by rewriting \eqref{VandermondeDefinition} as
\begin{equation}
  \label{eq:appendix_charpolynom:1}
  \begin{bmatrix}
    1 & k_1^2 & k_1^4 &\hdots & k_1^{2N-2} & 0\\
    1 & k_2^2 & k_2^4 &\hdots & k_2^{2N-2} & 0\\ 
\vdots& \vdots& \vdots&       & \vdots     & \vdots\\
    1 & k_N^2 & k_N^4 &\hdots & k_N^{2N-2} & 0\\
    0 & 0     & 0     &\hdots & 0          & 1
  \end{bmatrix} 
  \begin{bmatrix} 
    \zeta_0 \\ \zeta_1 \\ \vdots \\ \zeta_{N-1} \\ 1 
  \end{bmatrix} =
  \begin{bmatrix} 
    -k_1^{2N} \\ -k_2^{2N} \\ \vdots \\ -k_N^{2N} \\ 1 
  \end{bmatrix}
\end{equation}
For $m\in\mathbb{N}$ we define $P_m \DEF [p_1,p_2,\ldots,p_m]$ with $p_i \in \{0,1\}$ for $i=1,\ldots,m$. We set $|P_m| = \sum_{i=1}^m p_i$  equal to the number of 1's in $P_m$. Using Gaussian elimination the solution of \eqref{eq:appendix_charpolynom:1} is found to be
\begin{displaymath}
  Z = 
  \begin{bmatrix}
    (-1)^{N-0} \sum_{|P_N|=N-0}{k_1^{2p_1} k_2^{2p_2} \ldots k_N^{2p_N}}\\
    (-1)^{N-1} \sum_{|P_N|=N-1}{k_1^{2p_1} k_2^{2p_2} \ldots k_N^{2p_N}}\\
    \vdots \\
    (-1)^{1} \sum_{|P_N|=1}{k_1^{2p_1} k_2^{2p_2} \ldots k_N^{2p_N}}\\
    1
  \end{bmatrix}
 \end{displaymath}
From the proof of Lemma \ref{lmm:ODEEquivalence} we recall the decomposition 
\begin{equation}
  \label{eq:appendix_charpolynom:2}
  M^T = e^{\lambda\tau_0}(\lambda + \alpha)I + 2\Xi
\end{equation}
where $I$ is the $(N+1) \times (N+1)$ identity matrix and $\Xi$ was defined in the proof of Lemma \ref{lmm:ODEEquivalence}. Expanding the bilinear forms in the matrix $\Xi$ and moving the summation in front of the matrix yields
\begin{equation}
  \label{eq:appendix_charpolynom:3}
   \Xi = \sum_{i=1}^{N} c_i k_i \Xi_i, \qquad \Xi_i \DEF 
   \begin{bmatrix}
    0      & 1 & k_i^2 & \hdots & k_i^{2(N-1)} \\
    0      & 0 & 1 & \hdots & k_i^{2(N-2)} \\
    \vdots & & \ddots & \ddots & \vdots \\
    0      & & & 0 & 1 \\
    0      & \ldots  & \hdots  & \hdots & 0
  \end{bmatrix}
\end{equation}
Now substitute \eqref{eq:appendix_charpolynom:2} with \eqref{eq:appendix_charpolynom:3} into \eqref{Pdef} to obtain
\begin{align}
  \mathcal{P}(\rho) &= r^T \Bigl[e^{\lambda\tau_0}(\lambda+\alpha)\Id + 2\sum_{i=1}^{N}{c_i k_i \Xi_i}\Bigr] Z\nonumber\\
  &= e^{\lambda\tau_0}(\lambda+\alpha) 
  \begin{bmatrix} 
    1 & \rho^2 & \rho^4 & \hdots & \rho^{2N} 
  \end{bmatrix} 
  Z + 2r^T \sum_{i=1}^{N}{c_i k_i \Xi_i} Z\label{eq:appendix_charpolynom:4}
\end{align}
We observe that on the one hand,
\begin{displaymath}
  e^{\lambda\tau_0}(\lambda+\alpha) 
  \begin{bmatrix} 
    1 & \rho^2 & \rho^4 & \hdots & \rho^{2N} 
  \end{bmatrix} 
  Z = e^{\lambda\tau_0}(\lambda+\alpha) \prod_{i=1}^{N} (\rho^2-k_i^2)
\end{displaymath}
while on the other hand,
\begin{align*}
  &r^T \sum_{i=1}^{N}{c_i k_i \Xi_i} Z =\\
  &\sum_{i=1}^{N}{c_i k_i 
  \begin{bmatrix} 
    1 & \rho^2 & \rho^4 & \hdots & \rho^{2N} 
  \end{bmatrix} 
  \begin{bmatrix}
    (-1)^{N-1} \sum_{|P_{N-1}|=N-1}{k_1^{2p_1} k_2^{2p_2} \hdots k_{i-1}^{2p_{i-1}}k_{i+1}^{2p_i}\hdots k_N^{2p_{N-1}}}\\
    (-1)^{N-2} \sum_{|P_{N-1}|=N-2}{k_1^{2p_1} k_2^{2p_2} \hdots k_{i-1}^{2p_{i-1}}k_{i+1}^{2p_i}\hdots k_N^{2p_{N-1}}}\\
    \vdots\\
    (-1)^{1} \sum_{|P_{N-1}|=1}{k_1^{2p_1} k_2^{2p_2} \hdots k_{i-1}^{2p_{i-1}}k_{i+1}^{2p_i}\hdots k_N^{2p_{N-1}}}\\
    1\\
    0
  \end{bmatrix}}
\end{align*}
Hence by \eqref{eq:appendix_charpolynom:4} it follows that 
\begin{displaymath}
  \mathcal{P}(\rho) = e^{\lambda\tau_0}(\lambda+\alpha) \prod_{j=1}^{N} (\rho^2-k_j^2) + 2\sum_{i=1}^{N} c_i k_i \prod_{\substack{j=1 \\ j\neq i}}^{N} (\rho^2-k_j^2)
\end{displaymath}
which is equivalent to \eqref{ODE01CharPolynom2}, in the sense that the two polynomials have identical roots. Hence the proof is complete.


\bibliographystyle{amsplain}
\bibliography{neuralfield}

\end{document}